\documentclass[reqno,11pt]{amsart}
\usepackage{amscd,amssymb}
\usepackage{amsmath}
\usepackage{color}
\usepackage{hyperref}
\usepackage{comment}
\usepackage{tikz}
\usetikzlibrary{matrix}
\usepackage{marginnote}
\usepackage{enumitem}

\numberwithin{equation}{section}

\theoremstyle{plain}
\newtheorem{theorem}[subsection]{Theorem}
\newtheorem{lemma}[subsection]{Lemma}
\newtheorem{corollary}[subsection]{Corollary}
\newtheorem{proposition}[subsection]{Proposition}

\newtheorem{assumption}[subsection]{Assumption}

\newtheorem{definition}[subsection]{Definition}

\theoremstyle{definition}

\theoremstyle{remark}
\newtheorem{remark}[subsection]{Remark}

\usepackage{amsmath,amsfonts,latexsym}

\newcommand{\F}{\mathbb{F}}

\newcommand{\tensor}{\otimes}

\DeclareMathOperator{\Gr}{Gr}

\DeclareMathOperator{\id}{id}
\DeclareMathOperator{\Ker}{Ker}

\DeclareMathOperator{\supp}{supp}   
\DeclareMathOperator{\im}{im}      
\DeclareMathOperator{\End}{End}    

\DeclareMathOperator{\tr}{tr}

\DeclareMathOperator{\pr}{pr}

\DeclareMathOperator{\ind}{ind}

\newcommand{\forget}[1]{}

\newcommand{\innerprod}[1]{\langle #1 \rangle}

\allowdisplaybreaks[2]

\newcommand{\RR}{\mathbb{R}}
\newcommand{\CC}{\mathbb{C}}
\newcommand{\ZZ}{\mathbb{Z}}

\newcommand{\Id}{\operatorname{Id}}
\newcommand{\NGam}{\mathcal{N}\Gamma}
\newcommand{\<}{\langle}
\renewcommand{\>}{\rangle}
\newcommand{\widehatD}{\widehat{D}}

\begin{document}
\title{Callias-type operators in von Neumann algebras}

\author{Maxim Braverman${}^\dag$}
\address{Department of Mathematics,
Northeastern University,
Boston, MA 02115,
USA}

\email{maximbraverman@neu.edu}
\urladdr{www.math.neu.edu/~braverman/}

\author{Simone Cecchini}
\address{Department of Mathematics,
Northeastern University,
Boston, MA 02115,
USA}

\email{cecchini.s@husky.neu.edu}

\thanks{${}^\dag$Supported in part by the NSF grant DMS-1005888.}


\begin{abstract}
We study differential operators on complete Riemannian manifolds which act on sections of a bundle of finite type modules over a von Neumann algebra with a trace. We prove a relative index and a Callias-type index theorems for von Neumann indexes of such operators. We apply these results to obtain a version of Atiyah's $L^2$-index theorem, which states that the index of a Callias-type operator on a non-compact manifold $M$ is equal to the $\Gamma$-index of its lift to a Galois cover of $M$.  We also prove the cobordism invariance of the index of Callias-type operators. In particular, we give a new proof of the cobordism invariance of the von Neumann index of operators on compact manifolds. 
\end{abstract}
\maketitle

\maketitle

\section{Introduction}\label{S:intro}

Index theory of operators acting on bundles of modules over a von Neumann algebra was initiated by Atiyah's \cite{Atiyah76vN}. Atiyah studied the special case of a Galois cover of a compact manifold $M$ with a covering group $\Gamma$. More specifically, let $D$ be an elliptic operator acting on a finite dimensional vector bundle over a compact manifold $M$ and let $\widetilde{M}$ be a Galois cover of $M$ with covering group $\Gamma$. We denote by $\mathcal{N}\Gamma$ the von Neumann algebra of $\Gamma$.  The lift $\widetilde{D}$ of $D$ to $\widetilde{M}$  can be viewed as an operator on $M$ with values in a bundle of $\mathcal{N}\Gamma$-modules. Atiyah studied the $\Gamma$-index of this operator and showed that it is equal to the index of $D$. The index theorem for an operator acting on a general bundle of von Neumann modules over a compact manifold was obtained by Schick \cite{Schick05l2}. 

In this paper we investigate the von Neumann index of Callias-type operators on  {\em non-compact manifolds}. 
A Callias-type operator on a complete Riemannian manifold is an operator of the form $B=D+\Phi$, where $D$ is a first order elliptic operator and $\Phi$ is a potential satisfying certain conditions. If $B$ acts on sections of a finite dimensional vector bundle, then the conditions on $\Phi$ guarantee that the operator $B$ is Fredholm. The study of such operators was initiated by Callias, \cite{Callias78}, and further developed by many authors, cf., for example, \cite{BottSeeley78}, \cite{BruningMoscovici}, \cite{Anghel1993}, \cite{Bunke1995}. Several generalizations and new applications of the Callias-type index theorem were obtained recently, cf. \cite{Kottke11}, \cite{CarvalhoNistor14}, \cite{Wimmer14}, \cite{Kottke15}, \cite{BravermanShi}.

The extension of the definition of Callias-type index to operators acting on sections of a bundle of modules over a von Neumann algebra was obtained in \cite{BrCe15}. In this paper we obtain several properties of this index.

\subsection{The relative index theorem}\label{SS:Irelative index}
The relative index theorem of Gromov and Lawson \cite{GromovLawson83} allows to use a ``cut and paste" procedure to compute the index. This theorem was reformulated for Callias-type operators by Anghel \cite{anghel1990} and Bunke \cite{Bunke1995}. Moreover, Bunke considered the case of $C^*$-algebra valued differential operators invertible at infinity. 

Our first result in this paper is a version of the relative index theorem for  operators acting on bundles of modules over a von Neumann algebra with finite trace. 
Notice, that our result does not follow directly from \cite{Bunke1995}, since it   is not clear whether the von Neumann index descends from the $K$-theoretical index used by Bunke. In particular, it is not clear whether our operators satisfy the conditions required in \cite{Bunke1995}.

\subsection{The Callias-type index theorem}\label{SS:ICallias index}
The Callias-type index theorem, \cite{Anghel1993}, \cite{Bunke1995}, states that the index of a Callias-type operator $B$ is equal to the index of a certain operator induced by the restriction of $B$ to a compact hypersurface. We show that this result continues to hold for the von Neumann index of operators acting on sections of a bundle of modules over a von Neumann algebra. 

Our proof uses the relative index theorem and is similar to the proof in the case of a finite dimensional vector bundles in  \cite{Anghel1993} and \cite{Bunke1995}. However, more care is needed due to the fact that the spectrum of $B$ is not discrete around $0$, see Section~\ref{SS:diffrence from C} for details.

\subsection{The cobordism invariance of the von Neumann index}\label{SS:Icobordism}
The cobordism invariance of the index was used in the original proof of the index theorem \cite{AtSinger63} (see also \cite{Palais65}). Since then many generalizations and new analytic proofs of the cobordism invariance of the index appeared in the literature, see for example \cite{Higson91},\cite{Br-cob},\cite{Br-index},\cite{Carvalho05},\cite{Carvalho12},\cite{Br-cobtr},\cite{Hilsum10},\cite{Br-index4proper}. The cobordism invariance of the index of Callias-type operators acting on finite dimensional vector bundles was recently obtained in \cite{BravermanShi}. 

In this paper we show that the cobordism invariance of the index can be obtained as an easy corollary of the Callias-type index theorem. By this method we prove the cobordism invariance of the von Neumann index both for Dirac-type operators on compact manifolds and for  Callias-type operators on non-compact manifolds.

\subsection{Atiyah's type $L^2$-index theorem}\label{SS:IL2index}
Let $M$ be a complete Riemannian manifold and let $B$ be a Callias-type operator on $M$, acting on sections of a  {\em finite dimensional} vector bundle $E\to M$. Let $\widetilde{M}$ be a Galois cover of $M$ with covering group $\Gamma$. We denote by $\widetilde{B}$ the lift of $B$ to $\widetilde{M}$. Then $\widetilde{B}$ is a $\Gamma$-equivariant operator and its $\Gamma$-index $\ind_\Gamma\widetilde{B}$ is well defined. We show that 
\[
	\ind_\Gamma\widetilde{B} \ =\  \ind B.
\]
For the case when the manifold $M$ is compact, this result was proven by Atiyah  \cite{Atiyah76vN}.

\subsection{}\label{SS:Icontent} The paper is organized as follows: in Section~\ref{S:main results} we formulate the main results of the paper. In Section~\ref{S:relative index theorem} we prove the relative index theorem. Sections~\ref{S:cylinder}--\ref{S:general case} are devoted to the proof of the Callias-type index theorem. In Section~\ref{S:cobordism} we prove the cobordism invariance of the index.  Finally in Section~\ref{S:Gamma index in odd dimension} we prove the Atiyah-type $L^2$-index theorem for Callias-type operators.


\section{The main results}\label{S:main results}

In this section we formulate the main results of the paper.


\subsection{Callias-type operators}\label{SS:Callias-type operators}
Let $A$ be a von Neumann algebra with a finite, faithful, normal  trace $\tau:A\to \CC$ and let $M$ be a complete Riemannian manifold. 
Assume  $E=E^+\oplus E^-$ is a $\ZZ_2$-graded $A$-Hilbert bundle of finite type over $M$ (for this notion we refer to \cite[Sections~8.1-8.3]{Schick05l2} and \cite[Section~3.5]{BrCe15}).
Let
\[
\mathcal{D}:=\left(\begin{array}{cc}0&D^-\\D^+&0\end{array}\right)
\]
be an elliptic differential operator acting on smooth sections of $E$, where 
\[
D^\pm:C^\infty(M,E^\pm)\rightarrow C^\infty(M,E^\mp)
\]
 are formally adjoint to each other.
We make the following
\begin{assumption}\label{A:uniformly bounded principal symbol}
The principal symbol of $\mathcal{D}$ is uniformly bounded from above, i.e. there exists a constant $b>0$ such that
\begin{equation}
\|\sigma(\mathcal{D})(x,\xi)\|\leq b\,|\xi|,\qquad\qquad\textrm{for all }x\in M,\ \xi\in T^\ast_xM.
\end{equation}
Here $|\xi|$ denotes the length of $\xi$ defined by the Riemannian metric on M,\/ $\sigma(D)(x,\xi): E^+_x\rightarrow E^-_x$ is the leading symbol of $\mathcal{D}$, and $\|\sigma(\mathcal{D})(x,\xi)\|$ is its operator norm.
\end{assumption}

By \cite{BrCe15}, the above assumption guarantees that the operator $\mathcal{D}$ is essentially self-adjoint. 

Let $F^+:E^+\rightarrow E^-$ be a morphism of $A$-Hilbert bundles and consider the odd self-adjoint endomorphism of $E$
\[
\mathcal{F}:=\left(\begin{array}{cc}0&F^-\\F^+&0\end{array}\right),
\]
where $F^-:E^-\rightarrow E^+$ is the formal adjoint of $F^+$.

\begin{definition}\label{D:admissible endomorphism}
The endomorphism $\mathcal{F}$ is called \emph{admissible} if
\begin{enumerate}
\item[(i)] the anticommutator 
\[
\{\mathcal{D},\mathcal{F}\}:=\mathcal{D}\circ\mathcal{F}+\mathcal{F}\circ\mathcal{D}=
\left(\begin{array}{cc}D^-F^++F^-D^+&0\\0&D^+F^-+F^+D^-\end{array}\right)
\]
 is an operator of order
 $0$;
\item[(ii)] there exist a constant $c>0$ and a compact set $K\subset M$ such that 
\begin{equation}\label{E:callias condition}
\mathcal{F}^2+\{\mathcal{D},\mathcal{F}\}\geq c,\ \ \ \ \ \ \ \ \ \ \ \ \textrm{ on }M\setminus K.
\end{equation}
\end{enumerate}
In this case, we say that the compact set $K$ is an \emph{essential support} of $\mathcal{F}$.
\end{definition}

\begin{definition}\label{D:Callias}
A {\em Callias-type operator} is an operator of the form
\begin{equation}\label{E:Callias}
	B\ :=\ \mathcal{D}+\mathcal{F}\ =\ \left(\begin{array}{cc}0&D^-+F^-\\D^++F^+&0\end{array}\right),
\end{equation}
where $\mathcal{F}$ is an admissible endomorphism. We set $B^\pm:=D^\pm+F^\pm$.
\end{definition}

\subsection{The $\tau$-index of a Callias-type operator}\label{SS:Callias index}

We use the Riemannian metric on $M$ and the inner product on the fibers of $E$ to define the $A$-Hilbert space $L^2(M,E)$ of square-integrable sections of $E$ and we regard $B$ as an $A$-linear unbounded operator on $L^2(M,E)$.
By \cite[Theorem~2.3]{BrCe15}, $B$ is essentially self-adjoint with initial domain $C^\infty_c(M,E)$.
We denote its closure by the same symbol $B$.

The  trace $\tau:A\to \CC$ on $A$ induces a dimension function 
\[
\dim_\tau:\Gr(L^2(M,E))\longrightarrow [0,\infty],
\]
where $\Gr(L^2(M,E))$ denotes the set of closed $A$-invariant subspaces of $L^2(M,E)$: for the construction of the function $\dim_\tau$ out of the trace $\tau$ see \cite[Sections~1.1.3~and~9.1.4]{lueck2002l2}.
By \cite[Theorem~2.19]{BrCe15}, the operator $B$ is $\tau$-Fredholm (for the notion of closed $\tau$-Fredholm operator we refer to  \cite{Breuer68Fredholm1} and \cite{Breuer69Fredholm2}).
In particular, this means that the $A$-Hilbert spaces $\Ker B^\pm$ have finite $\tau$-dimension and we define the $\tau$-index of $B$ by
\begin{equation}\label{E:tau index}
\ind_\tau B:=\dim_\tau\ker B^+-\dim_\tau\ker B^-.
\end{equation}


\subsection{A relative index theorem}\label{SS:A relative index theorem}

Suppose $E_j$, with $j=0,1$, are $\ZZ_2$-graded $A$-Hilbert bundles over complete Riemannian manifolds $M_j$ and $B_j=\mathcal{D}_j+\mathcal{F}_j$ are Callias-type operators.
Suppose $M_j=W_j\cup_{N_j} V_j$ are partitions of $M_j$, where $N_j$ are \emph{compact} hypersurfaces.
We make the following
\begin{assumption}\label{A:cut and paste}
There exist tubular neighborhoods $U(N_0)$, $U(N_1)$ respectively of $N_0$, $N_1$ and a diffeomorphism $\psi:U(N_0)\rightarrow U(N_1)$ such that:
\begin{enumerate}
\item[(i)] $\psi$ restricts to a diffeomorphism between $N_0$ and $N_1$;
\item[(ii)] there exists an isomorphism of $\ZZ_2$-graded $A$-Hilbert bundles $\Psi:E_0|_{U(N_0)}\rightarrow E_1|_{U(N_1)}$ covering $\psi$;
\item[(iii)] the restrictions $\mathcal{D}_0|_{U(N_0)}$ and $\mathcal{D}_1|_{U(N_1)}$ are conjugated through the isomorphism $\Psi$.
\end{enumerate}
\end{assumption}

We cut  $M_i$ along $N_i$ and use the map $\psi$ to glue the pieces together interchanging $V_0$ and $V_1$.
In this way we obtain the manifolds 
\[
	M_2:=W_0\cup_NV_1, \qquad M_3:=W_1\cup_NV_0,
\]
where $N\cong N_0\cong N_1$. We refer to $M_2$ and $M_3$ as {\em manifolds obtained from $M_0$ and $M_1$ by cutting and pasting}.

We use the map $\Psi$ to cut the bundles $E_0$, $E_1$ at $N_0$, $N_1$ and glue the pieces together interchanging $E_0|_{V_0}$ and $E_1|_{V_1}$.
With this procedure we obtain $\ZZ_2$-graded $A$-Hilbert bundles $E_2\rightarrow M_2$ and $E_3\rightarrow M_3$.
Use Condition~(iii) of Assumption~\ref{A:cut and paste} to construct an odd formally self-adjoint first order elliptic operator $\mathcal{D}_2$ acting on smooth sections of $E_2$ satisfying Assumption~\ref{A:uniformly bounded principal symbol} and such that $\mathcal{D}_2|_{W_0}=\mathcal{D}_0|_{W_0}$ and $\mathcal{D}_2|_{V_1}=\mathcal{D}_1|_{V_1}$.
In the same fashion, define an elliptic operator $\mathcal{D}_3$ acting on smooth sections of the bundle $E_3$.

Assume $\mathcal{F}_0$ and $\mathcal{F}_1$ are admissible endomorphisms respectively of the bundles $E_0$ and $E_1$. For $j=0,1$ we choose positive functions $\alpha_j,\,\beta_j\in C^\infty(M_j)$ such that:
\begin{enumerate}[label=\textbf{(C.\arabic*)},ref=C.\arabic*]
\item \label{C:cut-off1} $\supp\alpha_j\subset W_j\cup U(N_j)\textrm{ and }\supp\beta_j\subset V_j\cup U(N_j)$;
\item \label{C:cut-off2} $\alpha_j=1\textrm{ on }W_j\setminus U(N_j)\textrm{ and }\beta_j=1\textrm{ on }V_j\setminus U(N_j)$;
\item \label{C:cut-off3} $\alpha_j^2+\beta_j^2=1$.
\end{enumerate}
By construction, $E_2|_{W_0\cup\, U(N_0)}\cong E_0|_{W_0\cup \,U(N_0)}$ and $E_2|_{V_1\cup\, U(N_1)}\cong E_1|_{V_1\cup\, U(N_1)}$.
We use this identification and Condition~(\ref{C:cut-off1}) to define the endomorphism $\mathcal{F}_2:=\mathcal{F}_0\alpha_0+\mathcal{F}_1\beta_1$ of $E_2$. In the same way we define the endomorphism $\mathcal{F}_3:=\mathcal{F}_1\alpha_1+\mathcal{F}_0\beta_0$ of $E_3$.
Observe that $\mathcal{F}_2$ and $\mathcal{F}_3$ are admissible so that the operators $B_2:=\mathcal{D}_2+\mathcal{F}_2$ and $B_3:=\mathcal{D}_3+\mathcal{F}_3$ are of Callias-type. We often refer to $B_2$ and $B_3$ as {\em operators obtained from $B_0$ and $B_1$ by cutting and pasting}.

The first result of the paper is the following

\begin{theorem}[Relative index theorem in von Neumann setting]\label{T:relative index theorem in von neumann setting}
\[
	\ind_\tau B_0+\ind_\tau B_1\ =\ \ind_\tau B_2+\ind_\tau B_3.
\]
\end{theorem}

\begin{remark}
In the case $A=\CC$, this theorem was proved by Gromov and Lawson in \cite{GromovLawson83}. A $K$-theoretical version has been proved by Bunke in \cite{Bunke1995}. Our formulation of the relative index theorem is close to this last one.
The reason why we cannot apply directly \cite[Theorem~1.2]{Bunke1995} is that it is not clear whether the numerical index that we use in the present paper descends from a $K$-theoretical one.
In particular, it is not clear whether the operator $B$ is invertible at infinity in the sense of Bunke, cf. the discussion on page~258 of \cite{Bunke1995}.
\end{remark}


\subsection{Ungraded bundles and Callias-type operators.}\label{SS:the odd dimensional case}

We now describe a class of Callias-type operators constructed out of an \emph{ungraded} bundle.
Let $M$ be a complete odd-dimensional Riemannian manifold and let $\Sigma\rightarrow M$ be an ungraded $A$-Hilbert bundle over $M$. 
Suppose $D\colon C^\infty_c(M,\Sigma)\rightarrow C^\infty_c(M,\Sigma)$ is a formally self-adjoint first order elliptic operator satisfying Assumption~\ref{A:uniformly bounded principal symbol}. 
Let $\Phi$ be a self-adjoint endomorphism of $\Sigma$.

\begin{definition}\label{D:ungraded admissible endomorphism}
The endomorphism $\Phi$ is said to be \emph{admissible} for the pair $(\Sigma, D)$ if
\begin{enumerate}
\item[(i)] the commutator $[D,\Phi]:=D\Phi-\Phi D$ is an endomorphism of $\Sigma$;
\item[(ii)] there exist a constant $d>0$ and a compact set $K\subset M$ such that 
\begin{equation}\label{E:ungraded admissible endomorphism}
	\Phi^2(x)\ \geq\  d\ +\ \|[D,\Phi](x)\|,\qquad x\in M\setminus K,
\end{equation}
where $\|[D,\Phi](x)\|$ denotes the operator norm of the bounded operator $[D,\Phi](x):\Sigma_x\to \Sigma_x$.
\end{enumerate}
In this case we say that $K$ is an \emph{essential support} of $\Phi$ with respect to the pair $(\Sigma,D)$ and that $(\Sigma, D, \Phi)$ is an \emph{admissible ungraded triple} over $M$.
\end{definition}

Suppose the endomorphism $\Phi$ is admissible and regard $\Sigma\oplus\Sigma$ as a $\ZZ_2$-graded $A$-Hilbert bundle over $M$. Then the first order elliptic operator
\begin{equation}\label{E:Callias-type operator in odd dimension1}
\mathcal{D}:=\left(\begin{array}{cc}0&D\\D&0\end{array}\right)
\end{equation}
is odd with respect to this grading, formally self-adjoint and satisfies Assumption~\ref{A:uniformly bounded principal symbol}.
We use the endomorphism $\Phi$ to define the odd self-adjoint endomorphism of $\Sigma\oplus\Sigma$
\begin{equation}\label{E:Callias-type operator in odd dimension2}
\mathcal{F}_\Phi:=i\left(\begin{array}{cc}0&-\Phi\\\Phi&0\end{array}\right).
\end{equation}
Conditions (i) and (ii) of Definition~\ref{D:ungraded admissible endomorphism} imply that $\mathcal{F}_\Phi$ is admissible according to Definition~\ref{D:admissible endomorphism} and any essential support of $\Phi$ is also an essential support of $\mathcal{F}_\Phi$.

\begin{definition}\label{D:classical Callias-type operator}
The operator
\begin{equation}\label{E:classical Callias-type operator}
B_\Phi:=\mathcal{D}+\mathcal{F}_\Phi=\left(\begin{array}{cc}0&D-i\Phi\\D+i\Phi&0\end{array}\right)
\end{equation}
is called the \emph{Callias-type operator associated with the triple $(\Sigma, D, \Phi)$}.
\end{definition}
The equation \eqref{E:tau index} for the Callias index takes in this case the form
\begin{equation}\label{E:nongraded index}
	\ind_\tau B_\Phi\ = \ \dim_\tau\ker (D+i\Phi)\ - \ \dim_\tau\ker (D-i\Phi).
\end{equation}
In particular, we obtain
\begin{equation}\label{E:Phi=-Phi}
	\ind_\tau B_{-\Phi}\ = \ -\ind_\tau B_\Phi.
\end{equation}
Since the index of the Callias-type operator on a {\em compact} manifold is independent of $\Phi$ we obtain the following

\begin{proposition}\label{P:compact index}
The index of a Callias-type operator associated with an ungraded triple $(\Sigma,D,\Phi)$ over a {\em compact} manifold is equal to zero. 
\end{proposition}

\begin{remark}
The operator \eqref{E:classical Callias-type operator} is a Callias-type operator in the sense of Definition~\ref{D:Callias}. Notice, however, that not every Callias-type operator is  associated with some triple $(\Sigma, D, \Phi)$. In fact, an operator \eqref{E:Callias} is associated with such a triple  if and only if there exists an isomorphism $E^+\to E^-$ which commutes with the leading symbol of $\mathcal{D}$ and anticommutes with the admissible endomorphism $\mathcal{F}$. 
\end{remark}

In the case $A=\CC$, Callias-type operators associated with a triple $(\Sigma,D,\Phi)$ were extensively studied, cf. for example, \cite{Callias78}, \cite{BruningMoscovici} and \cite{Anghel1993}.


\subsection{Callias-type Theorem.}\label{SS:Callias-type theorem}
Let $M$ be a complete oriented $n$-dimensional Riemannian manifold. Let $\Sigma\rightarrow M$ be an ungraded Dirac $A$-Hilbert bundle over $M$.
This means that $\Sigma$ is an $A$-Hilbert bundle of finite type endowed with a Clifford action $c:T^*M\to \End_A(\Sigma)$ of the cotangent bundle and a metric connection  $\nabla^\Sigma$ compatible with the inner product of the fibers and satisfying the Leibniz rule
(for the theory of connections on $A$-Hilbert bundles we refer to \cite[Sections~8.1-8.3]{Schick05l2}).
When $A=\CC$, our definition coincides with the classical notion of a Dirac bundle (see \cite[Definition~II.5.2]{lawson1989spin}). 

An operator $D:C^\infty_0(M,\Sigma)\to C^\infty(M,\Sigma)$ is called a {\em (generalized) Dirac operator} if 
\[
	[D,f]\ =\  c(df)\qquad \text{for all} \quad f\in C^\infty(M).
\]
Let $(e_1,\ldots,e_n)$ be an orthonormal basis of the tangent bundle, and $(e^1,\ldots,e^n)$ be the dual basis of $T^*M$. By \cite[\S3.3]{BeGeVe}, there exists $V\in \End_A(\Sigma)$  such that 
\begin{equation}\label{E:definition of Dirac operator}
	D\ = \ \sum_{i=1}^n\,c(e^i)\,\nabla^{\Sigma}_{e_i} \ + \ V.
\end{equation}
In this situation we refer to $D$ as the {\em Dirac operator associated with the connection $\nabla^\Sigma$ and the potential $V$}.

It follows from \eqref{E:definition of Dirac operator} that $D$ satisfies Assumption~\ref{A:uniformly bounded principal symbol}.
Suppose $\Phi\in\End_A(\Sigma)$ is admissible. Note that Condition (i) of Definition~\ref{D:admissible endomorphism} is equivalent in this case to the condition that 
\begin{equation}\label{E:[Phi,c]}
	[c(\xi),\Phi]\ = \ 0, \qquad\text{for all}\quad \xi\in T^*M.
\end{equation}
We denote by $B_\Phi$ the Callias-type operator associated with the triple $(\Sigma, D, \Phi)$.

Suppose that there is a partition $M=M_-\cup_N M_+$, where $N=M_-\cap M_+$ is a smooth compact hypersurface and $M_-$ is a compact submanifold, whose interior contains an essential support of $\Phi$.
Let $\Sigma_N$ be the  restriction of $\Sigma$ to $N\subset M$.
By Condition (ii) of Definition~\ref{D:ungraded admissible endomorphism}, zero is not in the spectrum of $\Phi(x)$ for all $x\in N$. Therefore 
\begin{equation}\label{E:decomposition of Sigma_N}
\Sigma_N=\Sigma_{N+}\oplus \Sigma_{N-},
\end{equation}
where the fiber of $\Sigma_{N+}$ (resp. $\Sigma_{N-}$) over $x\in N$ is the image of the spectral projection of $\Phi_N(x)$ corresponding to the interval $(0,\infty)$ (resp. $(-\infty,0)$).
By \eqref{E:[Phi,c]} the endomorphism $\Phi$ commutes with the Clifford multiplication. Hence $c(\xi):\Sigma_{N\pm}\to \Sigma_{N\pm}$ for all $\xi\in T^*M$. It follows that both bundles, $\Sigma_{N+}$ and $\Sigma_{N-}$, inherit  the Clifford action of $T^*M$.

In particular, the Clifford multiplication by the unit normal vector field pointing at the direction of $M_+$ defines an endomorphism $\gamma:\Sigma_{N\pm}\to \Sigma_{N\pm}$.
Since $\gamma^2=-1$, the endomorphism $\alpha:=-i \gamma$ induces a grading 
\begin{equation}\label{E:grading on SigmaN}
	\Sigma_{N\pm}\ = \ \Sigma_{N\pm}^+\oplus \Sigma_{N\pm}^-,
\end{equation}
where $\Sigma_{N\pm}$ is the span of the eigenvectors of $\alpha$ with eigenvalues $\pm1$.

We use the Riemannian metric on $M$ to identify $T^*N$ with a subspace of $T^*M$. Then the Clifford action of $T^*N$ on $\Sigma_{N\pm}$ is graded with respect to this grading, i.e. $c(\xi):\Sigma_\pm\to \Sigma_\mp$ for all $\xi\in T^*N$. 

Let $\nabla^{\Sigma_N}$ be the connection on $\Sigma_N$ obtained by restricting  the connection on $\Sigma$. It does not, in general, preserve decomposition~\eqref{E:decomposition of Sigma_N}. 
We define a connection $\nabla^{\Sigma_{N\pm}}$ on the bundle $\Sigma_{N\pm}$ by 
\begin{equation}\label{E:nSigmaN+}
	\nabla^{\Sigma_{N\pm}} s^\pm\ =\ 
	\pr_{\Sigma_{N\pm}}\left(\nabla^{\Sigma_N} s^\pm\right),
	\qquad s^\pm\in C^\infty(N,\Sigma_{N\pm}),
\end{equation}
where $\pr_{\Sigma_{N\pm}}$ is the projection onto the bundle $T^\ast N\tensor\Sigma_{N\pm}$. Since $\gamma$ commutes with both, the connection $\nabla^{\Sigma_N}$ and the projection $\pr_{\Sigma_{N\pm}}$, it also commutes with the connection $\nabla^{\Sigma_{N\pm}}$. In other words, $\nabla^{\Sigma_{N\pm}}$ is a graded connection with respect to the grading \eqref{E:grading on SigmaN}. 
By \cite[Lemma~2.7]{anghel1990} (see also Section~\ref{SS:cylindrical end}),  it is also compatible with the inner product of the fibers of $\Sigma_{N\pm}$ and satisfies the Liebniz rule.
Therefore, $\Sigma_{N\pm}$ carries a $\ZZ_2$-graded Dirac $A$-Hilbert bundle structure.

We denote by  $D_{N}:= D_{N+}\oplus D_{N-}$ the Dirac operator on $N$ associated with the connection $\nabla^{\Sigma_{N+}}\oplus\nabla^{\Sigma_{N-}}$ and the zero potential. Then
\begin{equation}\label{E:DN}
	D_{N\pm}\ := \ \sum_{i=1}^{n-1} c(e^i)\,\nabla^{\Sigma_{N\pm}}_{e_i},
\end{equation}
where $(e_1,\ldots,e_{n-1})$ is an orthonormal frame of $TN$ and $(e^1,\ldots,e^{n-1})$ is the dual frame of $T^*N$. The operators $D_{N\pm}$ are odd  with respect to the grading \eqref{E:grading on SigmaN}, i.e. they have the form 
\[
	D_{N\pm}\ = \ \begin{pmatrix}
	0&D_{N\pm}^-\\D_{N\pm}^+&0
	\end{pmatrix},
\]
where $D_{N\pm}^+$ (respectively $D_{N\pm}^-$) is the restriction of $D_{N\pm}$ to $\Sigma_{N\pm}^+$ (respectively $\Sigma_{N\pm}^-$). It is a classical fact that $D_{N\pm}$ are $\tau$-Fredholm (see for instance \cite[Section~7]{Schick05l2}) so that they have a well-defined $\tau$-index
\begin{equation}\label{E:indexDNpm}
	\ind_\tau D_{N\pm}\ = \ \dim_\tau\ker D_{N\pm}^+ \ -\
	\dim_\tau\ker D_{N\pm}^-. 
\end{equation}
The next theorem is the first main result of this paper. 
It generalizes \cite[Theorem~1.5]{Anghel1993} and \cite[Theorem~2.9]{Bunke1995} to the von Neumann setting.

\begin{theorem}[Callias-type theorem in von Neumann setting]\label{T:computation of odd-dimensional case}
\[
	\ind_\tau B_\Phi\ =\ \ind_\tau D_{N+}.
\]
\end{theorem}

From this theorem and \eqref{E:Phi=-Phi} we immediately obtain 
\begin{corollary}\label{C:Phi=-Phi}
$\ind_\tau{}B_\Phi= -\ind_\tau{}D_{N-}$. Hence, 
\begin{equation}\label{E:N+=-N-}
	\ind_\tau{}D_{N+}\ = \ -\ind_\tau{}D_{N-}
\end{equation}
\end{corollary}

\begin{remark}\label{R:Phi=-Phi}
Equality \eqref{E:N+=-N-} might look a bit surprising, since at the first glance the bundles $\Sigma_+$ and $\Sigma_-$ might look completely unrelated. However, the fact that both operators, $D_{N+}$ and $D_{N-}$, are induced by the same operator $D$ puts a strong relation between them. In fact, the direct sum $D_{N+}\oplus{}D_{N-}$ is {\em cobordant to 0} in the sense of \cite{Br-cob}, and the cobordism is given by the operator $D$. Thus the equality \eqref{E:N+=-N-} can be viewed as a version of the cobordism invariance of the index. In fact, it is, in a sense, equivalent to the cobordism invariance, as explained in the next subsection.
\end{remark}

If $\Phi$ is an admissible function for $\Sigma$ then so is $\lambda{}\Phi$ for all $\lambda>1$. As an immediate corollary of Theorem~\ref{T:computation of odd-dimensional case} we obtain the following

\begin{corollary}\label{C:lambdaPhi}
The index $\ind_\tau{}B_{\lambda \Phi}$ is independent of $\lambda\ge 1$. 
\end{corollary} 
\begin{remark}\label{R:lambdaPhi}
If the endomorphism $\Phi$ is bounded then the domain of the operator $B_{\lambda \Phi}$ is independent of $\lambda$ and the corollary follows directly from the stability of the $\tau$-index, cf. \cite[Lemma~7.3]{BrCe15}. However, if $\Phi$ is not bounded the domain of $B_{\lambda \Phi}$ depends on $\lambda$ and the corollary is not {\em a priori} obvious. 
\end{remark}

\subsection{The cobordism invariance}\label{SS:cobordism}
We now introduce a class of non-compact cobordisms similar to those considered in \cite{GGK96},\cite{GGK-book},\cite{Br-index},\cite{BrCano14},\cite{BravermanShi}. Then we show that the index is invariant under this type of cobordisms. To make the exposition simpler we will only discuss the {\em null-cobordism} (i.e. the cobordism between a manifold and the empty set). We show that the index of a null-cobordant operator vanishes. A standard argument, cf. for example \cite[Remark~2.10]{BravermanShi}, shows that this is equivalent to the cobordism invariance of the index.

\newcommand{\oSigma}{\overline{\Sigma}}
\newcommand{\oD}{\overline{D}}
\newcommand{\oPhi}{\overline{\Phi}}
\begin{definition}\label{D:cobordismSigma}
Let $\Sigma\to M$ be a Dirac bundles over a complete oriented Riemannian manifold $M$. Let $D$  be a (generalized) Dirac operator on $\Sigma$.

A triple $(W,\oSigma,\oD)$, where $W$ is a complete manifold with boundary $\partial{W}\simeq M$, $\oSigma$ is a Dirac $A$-bundle over $W$, and $\oD$ is a Dirac operator on $\oSigma$, is called a {\em null-cobordism} of $D$ if the following conditions are satisfied:
\begin{enumerate}[label=(\roman*)]
\item 
There is an open neighborhood $U$ of\/ $\partial{W}$ and a metric-preserving diffeomorphism
\begin{equation}\label{E:diffeo}
    \phi: M\times(-\epsilon,0] \ \to \ U.
\end{equation}
\item 
Let $\widehat{\Sigma}\simeq \Sigma\times\RR$ denote the lift of $\Sigma$ to $M\times\RR$. Then the bundle $\widehat{\Sigma}\oplus\widehat{\Sigma}$ has a natural structure of a Dirac bundle over $M\times\RR$, cf. Section~\ref{SS:cylindrical end}. We assume that the restriction of $\oSigma$ to $M\times(-\epsilon,0]\subset U$ is isomorphic to $\widehat{\Sigma}\oplus\widehat{\Sigma}$.
\item
Let $t$ denote the coordinate on $\RR$. Consider the operator  
\[
	\widehat{D}\ := \ 
	\begin{pmatrix}
	i\frac{\partial}{\partial t}&D\\D&-i\frac{\partial}{\partial t}
	\end{pmatrix}:\,
	C^\infty_0(M\times\RR,\widehat{\Sigma}\oplus \widehat{\Sigma})
	 \ \to \ 
	C^\infty_0(M_i\times\RR,\widehat{\Sigma}\oplus \widehat{\Sigma}).
\]
Then the restriction of $\oD$ to $M\times(-\epsilon,0]$ is equal to the restriction of $\widehat{D}$ to the same set. 
\end{enumerate}  
\end{definition}

Notice, that the notion of non-compact cobordism is not very useful by itself. For example for any triple $(M,\Sigma,D)$ the pair $(M\times[0,\infty),\Sigma\times[0,\infty),\widehat{D})$ is a null-cobordism of $(M,\Sigma,D)$. However, if one considers cobordisms carrying some extra structure, like in the next definition, the notion of non-compact cobordism becomes useful and non-trivial.

\begin{definition}\label{D:calliascob}
Let $\Sigma\to M$ be a Dirac bundle over a complete oriented Riemannian manifolds $M$ and let $D$ be a (generalized) Dirac operator on $\Sigma$.  Suppose $\Phi\in \End_A(\Sigma)$ is an  admissible endomorphism for $(\Sigma,D)$. Let $B_{\Phi}$  be the corresponding Callias-type operator.  A {\em null-cobordism} of  $B_{\Phi}$  consists of the following data:
\begin{enumerate}[label=(\roman*)]
\item 
A null-cobordism $(W,\oSigma,\oD)$ of $D$;
\item 
A bundle map $\oPhi\in \End_A(\oSigma)$, such that the restriction of $\oPhi$ to $M\subset \partial{}W$ is equal to 
\[
 \begin{pmatrix}
	\Phi&0\\0&\Phi 
\end{pmatrix} \ \in\ \End_A\big(\widehat{\Sigma}\oplus\widehat{\Sigma}\big)
\] 
and there exist a constant $d>0$ and  a compact set $K\subset W$ such that Inequality~\eqref{E:ungraded admissible endomorphism} holds for all $x\in W\setminus{}K$. 
\end{enumerate}
The operator $B_\Phi$ is called {\em null-cobordant} if there exists a null-cobordism of $B_\Phi$. 

%
\end{definition}

\begin{theorem}\label{T:cobordism invariance}
The $\tau$-index of a null-cobordant Callias-type operator is equal to 0. 
\end{theorem}

In the case when the manifold $M$ is compact the theorem implies the cobordism invariance of the $\tau$-index of  Dirac operators on compact manifolds. In fact, as a part of the proof of Theorem~\ref{T:cobordism invariance}, in Section~\ref{S:cobordism} we give a new proof of the cobordism invariance of the $\tau$-index on compact manifolds. 

\begin{remark}\label{R:equivalence of theorems}
In this paper we first prove the relative index and the Callias-type index theorems and then use them to prove the cobordism invariance of the index. The opposite order of arguments is also possible. In fact, for the case $A=\CC$, i.e. for operators acting on finite dimensional vector bundles, the direct proof of the cobordism invariance of the Callias index was given in \cite{BravermanShi}. The proof in \cite{BravermanShi} can be adapted to our current situation using a method similar to the one used in Section~\ref{S:cylinder}. Then, as it is explained in  \cite{BravermanShi}, the cobordism invariance of the index implies the relative index theorem. 
\end{remark}


\subsection{The{ $\Gamma$}-index Theorem}\label{SS:the Gamma-index theorem}

Let $M$ be a complete oriented Riemannian manifold and let $S\rightarrow M$ be an ungraded finite dimensional Dirac bundle.%
\footnote{In other words in this section we assume that $A=\CC$. We use the notation $S$ (rather than $\Sigma$) for the Dirac bundle, to stress the difference from the other sections where $\Sigma$ was a Hilbert $A$-bundle.}
Let $D$ be a Dirac operator on $S$, cf. \eqref{E:definition of Dirac operator}. Let $\Phi$ be an admissible endomorphism of $\Sigma$.
We also suppose that $\widetilde{M}$ is a Galois cover of $M$ with deck transformation group $\Gamma$.
The bundle $S$ lifts to a bundle $\widetilde{S}\rightarrow\widetilde{M}$ and the Callias-type operator $B_\Phi:=\mathcal{D}+\mathcal{F}_\Phi$ defined by~(\ref{E:classical Callias-type operator}) lifts to a $\Gamma$-equivariant first order elliptic operator $\widetilde{B_\Phi}:=\widetilde{\mathcal{D}}+\widetilde{\mathcal{F}_\Phi}$ acting on smooth sections of $\widetilde{S}\oplus\widetilde{S}$.

The $\Gamma$-action on $C^\infty_c(\widetilde{M},\widetilde{S})$ induces an action of $\Gamma$ on the Hilbert space $L^2(\widetilde{M},\widetilde{S})$.
Fix a closed $\Gamma$-invariant subspace $L$ of $L^2(\widetilde{M},\widetilde{S})$.
Denote by $P_L$ the orthogonal projection onto $L$ and by $K_L(\cdot,\cdot)$ the Schwartz kernel of $P_L$.
For $x\in\widetilde{M}$,  $K_L(x,x)$ is an endomorphism of the finite dimensional complex vector space $\widetilde{S}_x$ and has a well-defined trace $\tr K_L(x,x)\in [0,\infty]$. 
The $\Gamma$-dimension of $L$ is defined by the formula
\begin{equation}\label{E:Gamma-dimension}
\dim_\Gamma (L):=\int_\Omega \tr K_L(x,x)\,d\widetilde{\mu}(x),
\end{equation}
where $\Omega\subset\widetilde{M}$ is a fundamental domain for the action of $\Gamma$ and 
$d\widetilde{\mu}$ is the positive $\Gamma$-invariant measure induced by the Riemannian metric on $M$ lifted to $\widetilde{M}$.
Equation (\ref{E:Gamma-dimension}) defines a dimension function on the set of closed $\Gamma$-invariant subspaces of $L^2(\widetilde{M},\widetilde{S})$. We say that $L$ has a {\em finite von Neuman dimension} if $\dim_\Gamma(L)<\infty$.

In Section~\ref{SS:Galois covers}, we show that $\widetilde{B_\Phi}$ can be regarded as a Callias-type operator constructed out of  a suitable ungraded Dirac $\NGam$-bundle $W$, where $\NGam$ is the group von Neumann algebra of $\Gamma$ (see Section~\ref{SS:Galois covers}).
As a consequence of \cite[Theorem~2.3]{BrCe15} we deduce the following

\begin{lemma}\label{L:Gamma index of Callias-type operators}
The closed $\Gamma$-invariant subspaces $\ker\, (\widetilde{D}\pm i\widetilde{\Phi})\subset L^2(\widetilde{M},\widetilde{V}^\pm)$ have finite von Neumann dimension.
\end{lemma}

Define the $\Gamma$-index of $\widetilde{B_\Phi}$ by
\[
\ind_\Gamma\widetilde{B_\Phi}:=\dim_\Gamma \ker (\widetilde{D}+ i\widetilde{\Phi})-\dim_\Gamma \ker (\widetilde{D}- i\widetilde{\Phi}).
\]

The second main result of this paper is the following generalization of Atiyah's $\Gamma$-index theorem (cf. \cite[Formula~(1.1)]{Atiyah76vN}) to operators of Callias-type.

\begin{theorem}[$\Gamma$-index theorem for Callias-type operators]\label{T:Gamma-index theorem in the odd-dimensional case}
Let $S$ be an ungraded Dirac bundle over a complete oriented Riemannian manifold $M$ and let $\Phi$ be an admissible self-adjoint endomorphism of $S$. 
If $\widetilde{M}$ is a Galois cover of $M$, then
\begin{equation}\label{E:gamma index theorem for Callias-type operators}
\ind_\Gamma \widetilde{B_\Phi}=\ind B_\Phi,
\end{equation}
where $\Gamma$ is the group of the deck transformations of $\widetilde{M}\rightarrow M$.
\end{theorem}

\subsection{Idea of the proof}
We prove Theorem~\ref{T:Gamma-index theorem in the odd-dimensional case} by reduction to the compact case.
We choose a compact hypersurface $N$ lying outside of an essential support of $\Phi$ and use Theorem~\ref{T:computation of odd-dimensional case} to construct a $\ZZ_2$-graded Dirac bundle $E\rightarrow N$ and a flat Dirac $\NGam$-bundle $W\rightarrow N$ such that

\begin{equation}\label{E:gamma index theorem for Callias-type operators2}
\ind B_\Phi=\ind D,\qquad \ind_\Gamma \widetilde{B_\Phi}=\ind_\tau D_W,
\end{equation}
where $D$ (resp. $D_W$) is the graded Dirac operator associated with $E$ (resp. $E\tensor W$).
The Atiah's $L^2$-index theorem  \cite{Atiyah76vN} (see also \cite{Schick05l2}) implies that 
\begin{equation}\label{E:gamma index theorem for Callias-type operators3}
\ind D\,=\,\ind_\tau D_W.
\end{equation}
Finally, formula~(\ref{E:gamma index theorem for Callias-type operators}) follows from (\ref{E:gamma index theorem for Callias-type operators2}) and (\ref{E:gamma index theorem for Callias-type operators3}).


\section{The relative index theorem}\label{S:relative index theorem}

In this section we prove Theorem~\ref{T:relative index theorem in von neumann setting}.
We adapt to the von Neumann setting the $K$-theoretical argument used by Bunke in \cite{Bunke1995}.


\subsection{Idea of the Proof}
Let $M_j$, $E_j$, $B_j:=\mathcal{D}_j+\mathcal{F}_j$ denote the same objects as in Section~\ref{SS:A relative index theorem} and set $M:=M_1\sqcup M_2\sqcup M_3\sqcup M_4$.
We use $E_1$, $E_2$, $E_3$ and $E_4$ to construct the bundle
\begin{equation}\label{E:bundle E_1 oplus E_2 oplus E_3^op oplus E_4^op}
E:=E_1\oplus E_2\oplus E_3^\textrm{op}\oplus E_4^\textrm{op}
\end{equation}
over $M$, where the superscript ``op" means that we consider the opposite grading on the fibers of $E_3$ and $E_4$ (see also Section~\ref{SS:construction of the operator B} for this notation).
Consider the Callias-type operator 
\begin{equation}\label{E:decomposition of B}
B:=B_1\oplus B_2\oplus B_3\oplus B_4
\end{equation}
acting on smooth sections of $E$. Then
\begin{equation}\label{E:tau-index of B}
\ind_\tau B=\ind_\tau B_1+\ind_\tau B_2-\ind_\tau B_3-\ind_\tau B_4.
\end{equation}
Theorem~\ref{T:relative index theorem in von neumann setting} is then equivalent to the equality $\ind_\tau B=0$.

For $j=1,2$, let $\alpha_j$, $\beta_j\in C^\infty(M_j)$ be the functions defined in Section ~\ref{SS:A relative index theorem}.
We use these functions to construct an {\em odd} self-adjoint $A$-equivariant bundle map $U\in C^\infty(M,\End_A(E))$ with the following properties:
\begin{enumerate}[label=\textbf{(U.\arabic*)},ref=U.\arabic*]
\item \label{U:property1}$U^2=\id_E$;
\item \label{U:property2} the anticommutator $\{B,U\}:=BU+UB$ is an even compactly supported endomorphism of the bundle $E$.
\end{enumerate}
By Condition~(\ref{U:property1}), $U^{-1}=U$. Since the degree of $U$ with respect to the grading on $E$ is odd,  i.e. $U:E^\pm\to E^\mp$, we have
\begin{equation}\label{E:indUBU=-indU}
\ind_\tau UBU=-\ind_\tau B.
\end{equation}
From Condition~(\ref{U:property2}) we deduce that $B+UBU$ is a compactly supported odd self-adjoint $A$-linear endomorphism of $E$.
Hence, by \cite[Theorem~2.21]{BrCe15} we get
\begin{equation}\label{E:indUBU=indU}
\ind_\tau UBU=\ind_\tau B.
\end{equation}
By Equations~(\ref{E:indUBU=-indU}) and (\ref{E:indUBU=indU}) we finally deduce that $\ind_\tau B=-\ind_\tau B$, from which the thesis follows.


\subsection{Construction of the operator $B$}\label{SS:construction of the operator B}
Let us first introduce some notation.
Given a $\ZZ_2$-graded $A$-module $Z=Z^+\oplus Z^-$, we denote by $Z^\textrm{op}$ the $A$-module $Z$ endowed with the opposite grading, i.e. $\left(Z^\textrm{op}\right)^\pm=Z^\mp$.
The bundles $E_1\rightarrow M_1$, $E_2\rightarrow M_2$, $E_3^\textrm{op}\rightarrow M_3$ and $E_4^\textrm{op}\rightarrow M_4$ define a $\ZZ_2$-graded $A$-Hilbert bundle $E\rightarrow M$ through Formula~(\ref{E:bundle E_1 oplus E_2 oplus E_3^op oplus E_4^op}).
Observe that
\[
C^\infty(M,E^\pm)=C^\infty(M_1,E_1^\pm)\oplus C^\infty(M_2,E_2^\pm)\oplus C^\infty(M_3,E_3^\mp)\oplus C^\infty(M_4,E_4^\mp).
\]
With respect to this decomposition, define the odd formally self-adjoint elliptic differential operator 
\[
\mathcal{D}=\left(\begin{array}{cc}0&D^-\\D^+&0\end{array}\right),
\]
where
\[
D^\pm=\left(\begin{array}{cccc}D_1^\pm&0&0&0\\0&D_2^\pm&0&0\\0&0&D_3^\mp&0\\0&0&0&D_4^\mp\end{array}\right)\colon C^\infty(M,E^\pm)\longrightarrow C^\infty(M,E^\mp).
\]
Here, $D_j^{\pm}=\mathcal{D}_j|_{C^\infty(M,E^\pm)}$. Notice that the operator $\mathcal{D}$ satisfies 
Assumption~\ref{A:uniformly bounded principal symbol}.
Define the self-adjoint endomorphism of $E$
\begin{equation}\label{E:relative index admissible endomorphism}
\mathcal{F}:=\left(\begin{array}{cc}0&F^-\\F^+&0\end{array}\right),
\end{equation}
where
\[
F^\pm=\left(\begin{array}{cccc}F_1^\pm&0&0&0\\0&F_2^\pm&0&0\\0&0&F_3^\mp&0\\0&0&0&F_4^\mp\end{array}\right)\colon E^\pm\longrightarrow E^\mp.
\]
Observe that $\mathcal{F}$ is admissible so that the operator 
\[
B:=\mathcal{D}+\mathcal{F}
\] 
is of Callias-type.
Denote by $H_j$ the $\ZZ_2$-graded $A$-Hilbert spaces $L^2(M_j,E_j)$ and set
\begin{equation}\label{E:A-Hilbert space H}
H:=H_1\oplus H_2\oplus H_3^\textrm{op}\oplus H_4^\textrm{op}.
\end{equation}
Notice that $H$ coincides with the space of $L^2$-sections of the $A$-Hilbert bundle $E$.
As in Section~\ref{SS:Callias-type operators} we use \cite{BrCe15} to conclude that the Callias-type operator $B$ as a closed odd self-adjoint $\tau$-Fredholm operator on $H$.
Similarly,  the Callias-type operators $B_j$ ($j=1,\ldots,4$) is a closed odd self-adjoint $\tau$-Fredholm operator on $H_j$.
From Decomposition~(\ref{E:A-Hilbert space H}) it follows that the operator $B$ has the form (\ref{E:decomposition of B})
so that it satisfies (\ref{E:tau-index of B}).
The construction of $B$ is complete.


\subsection{Construction of the bundle map $U$}

Recall that the functions $\alpha_j$, $\beta_j\in C^\infty(M_j)$ were defined in Section~\ref{SS:A relative index theorem}.
The function $\alpha_1$ is smooth and supported in $W_1\cup U(N_1)$.
Since $E_1|_{W_1\cup U(N_1)}\cong E_3|_{W_1\cup U(N_1)}$, multiplication by $\alpha_1$ defines a bundle map $a:E_1\rightarrow E_3$.
In the same fashion, we construct bundle maps $b:E_1\rightarrow E_4$, $c:E_2\rightarrow E_3$ and $d:E_2\rightarrow E_4$ by using respectively the functions $\beta_1$, $\beta_2$ and $\alpha_2$.
Observe that we can also regard $\alpha_1$ as a smooth function on $M_3$ with support in $W_1\cup U(N_1)$.
By using again the isomorphism $E_3|_{W_1\cup U(N_1)}\cong E_1|_{W_1\cup U(N_1)}$, multiplication by $\alpha_1$ gives a bundle map $E_3\rightarrow E_1$, that coincides with the map $a^\ast$ adjoint to $a$.
In the same way we see that the maps $b^\ast$, $c^\ast$ and $d^\ast$ are multiplication respectively by the functions $\beta_1$, $\beta_2$ and $\alpha_2$.
Condition (\ref{C:cut-off3}) of Section~\ref{SS:A relative index theorem} implies
\begin{equation}\label{E:computation with alphas and betas}
	aa^\ast+bb^\ast=1,\qquad cc^\ast+dd^\ast=1.
\end{equation}
We now assume that the functions $\alpha_j$  and $\beta_j$ are chosen in such a way that
\begin{equation}\label{E:identification of alpha_1 and alpha_2}
	\alpha_1|_{U(N_1)}=\alpha_2|_{U(N_2)},\qquad \beta_1|_{U(N_1)}=\beta_2|_{U(N_2)},
\end{equation}
where the neighborhoods $U(N_1)$ and $U(N_2)$ are identified through the isomorphism $\psi$ of Assumption~\ref{A:cut and paste}.
Define the odd $A$-linear self-adjoint operator
\[
	V\ :=\
	  \begin{pmatrix}
	     0&0&-a^\ast&-b^\ast\\
		0&0&-c^\ast&d^\ast\\
		a&c&0&0\\
		b&-d&0&0
	  \end{pmatrix}
	 \colon\, C^\infty(M,E^\pm)\ \longrightarrow\  C^\infty(M,E^\mp).
\]
We remark that the operator $V$ is odd because we use the opposite grading on $E_3$ and $E_4$.
Let $\epsilon$ be the grading operator on $E$, i.e. $\epsilon|_{E^\pm}=\pm 1$. 
Finally set
\[
	U\ :=\ \epsilon V.
\]
In the next two lemmas we show that the map $U$ has the required properties.

First, notice that since $V$ is odd with respect to the grading of $E$, we have
\begin{equation}
	\{\epsilon,U\}\ =\ \epsilon V\ +\ V\epsilon\ =\ 0, 
\end{equation}
i.e., $U$ is also an odd operator.

\begin{lemma}\label{L:U satisfies property 1}
$U$ satisfies Condition~(\ref{U:property1}).
\end{lemma}

\begin{proof}
Equations (\ref{E:computation with alphas and betas}) and (\ref{E:identification of alpha_1 and alpha_2}) together with Condition (\ref{C:cut-off1}) of Section~\ref{SS:A relative index theorem} imply
\[
	V\circ V\ =\ 
	\begin{pmatrix}
		-\alpha_1^2-\beta_1^2&-\alpha_1\beta_2+\beta_1\alpha_2&0&0\\
		-\beta_2\alpha_1+\alpha_2\beta_1&-\beta_2^2-\alpha_2^2&0&0\\
		0&0&-\alpha_1^2-\beta_2^2&-\alpha_1\beta_1+ \beta_2\alpha_2\\
		0&0&-\beta_1\alpha_1 +\alpha_2\beta_2&-\beta_1^2-\alpha_2^2
	\end{pmatrix}\ =\ -\id_E.
\]
Hence, $U^2= \epsilon{}V\epsilon{}V= -V^2=\id_E$.
\end{proof}


\begin{lemma}\label{L:U satisfies property 2}
$U$ satisfies condition (\ref{U:property2}).
\end{lemma}

\begin{proof}
Since both $B$ and $U$ are odd differential operators, the commutator $\{B,U\}$ is a differential operator of even degree.
It remains to show that $\{B,U\}$ is of order zero and compactly supported.
For $s^+\in C^\infty(M_1,E_1^+)$, we have
$$\begin{array}{rcl}
\left(UB+BU\right)\left(s^+\right)&=&aB_1^+s^+-B_3^+\left(as^+\right)\vspace{0.2cm}\\
&=&\alpha_1\left(D_1^++F_1^+\right)(s^+)-\left(D_1^++F_1^+\right)\left(\alpha_1s^+\right)\vspace{0.2cm}\\
&=&\alpha_1D_1^+s^+-D_1^+\left(\alpha_1s^+\right)\vspace{0.2cm}\\
&=&\left[\alpha_1,D_1^+\right]\left(s^+\right).\\
\end{array}$$
Conditions~(\ref{C:cut-off1}) and (\ref{C:cut-off2}) of Section~\ref{SS:A relative index theorem} imply that $\left[\alpha_1,D_1^+\right]$ is a bundle map supported in $U(N_1)$.
Since the closure of $U(N_1)$ is a compact set, the previous calculation shows that $\{U,B\}|_{C^\infty(M_1,E_1^+)}$ is a compactly supported homomorphism $E_1^+\rightarrow\left( E_3^\textrm{op}\right)^+$.
A similar argument shows that the homomorphism $\{U,B\}|_{C^\infty(M_j,E_j^\pm)}$ is compactly supported for all $j$.
\end{proof}


\subsection{Proof of Theorem~\ref{T:relative index theorem in von neumann setting}}
We need to show that Conditions (\ref{U:property1}) and (\ref{U:property2}) imply Equations~(\ref{E:indUBU=-indU}) and (\ref{E:indUBU=indU}).

We start with proving that Equation~(\ref{E:indUBU=-indU}) follows from Condition~(\ref{U:property1}).
Since the functions $\alpha_j$, $\beta_j$ are uniformly bounded, $U$ defines an $A$-equivariant bounded operator on the $A$-Hilbert space $L^2(M,E)$, that we denote by the same symbol $U$.
By Condition~(\ref{U:property1}), $U$ is unitary.
Since $B$ is $\tau$-Fredholm and $U$ is an $A$-equivariant isomorphism of $A$-Hilbert spaces, $UBU$ is $\tau$-Fredholm and the $\tau$-index of $ UBU$ is well-defined.
As a bundle map, $U$ is odd with respect to the grading of $E$, i.e. it is of the form
\[
U=\left(\begin{array}{cc}0&U^-\\U^+&0
\end{array}\right),
\]
where $U^\pm\colon C^\infty(M,E^\pm)\longrightarrow C^\infty(M,E^\mp)$.
It follows that $\left(UBU\right)^\pm=U^\pm B^\mp U^\pm$.
Since $U$ is an $A$-equivariant isomorphism, we deduce
\begin{equation}\label{E:kernel of UBU}
\Ker \left(UBU\right)^\pm\cong \Ker B^\mp,
\end{equation}
where the isomorphism is taken in the category of $A$-Hilbert spaces. 
Since $U$ is unitary, we use Borel functional calculus and from Equation~(\ref{E:kernel of UBU}) we deduce
\begin{equation}\label{E:indexl of UBU}
\dim_\tau \left(UBU\right)^\pm= \dim_\tau  B^\mp,
\end{equation}
from which Equation~(\ref{E:indUBU=-indU}) follows.

It remains to show that Equation~(\ref{E:indUBU=indU}) holds.
By Conditions~(\ref{U:property1}) and (\ref{U:property2}), the operator
\[
B+UBU=BU^2+UBU=\{B,U\}U
\]
is a compactly supported odd self-adjoint $A$-linear endomorphism of the bundle $E$.
It follows from \cite[Theorem~2.21]{BrCe15} that
\[
\ind_\tau B\ =\  \ind_\tau \left(B-(B+UBU)\right)\ =\ \ind_\tau \left(-UBU\right).
\]
Finally, Equation~(\ref{E:indUBU=indU}) follows by noticing that $\ind_\tau UBU=\ind_\tau \left(-UBU\right)$.
\hfill $\square$


\section{Analysis on the cylinder}\label{S:cylinder}

The next three sections are devoted to the proof of Theorem~\ref{T:computation of odd-dimensional case}.
In this section we define a particular Callias-type operator acting on a cylinder $N\times\RR$ with compact base $N$.
We call this operator the \emph{model operator} and show that Theorem~\ref{T:computation of odd-dimensional case} holds in this case.
In Section~\ref{S:cylindrical end} we consider the case of a manifold with cylindrical ends. We use the relative index theorem to reduce this case to the case of a  cylinder $N\times\RR$. Finally, in Section~\ref{S:general case} we again use the relative index theorem to reduce the general case to the case of a manifold with cylindrical ends. 


\subsection{The model operator}\label{SS:The model operator}

Let $E_N=E^+_N\oplus E^-_N$ be a $\ZZ_2$-graded Dirac $A$-Hilbert bundle over a {\em compact} oriented manifold $N$. Let $\nabla^E=\nabla^{E^+_N}\oplus\nabla^{E^-_N}$ be the connection on $E_N$ and let $D_N$
\[
	D_N\ =\ \left(\begin{array}{cc}0&D_N^-\\D^+_N&0\end{array}\right),
\]
be a Dirac operator associated with $\nabla^{E_N}$ and the zero potential, cf. Section~\ref{SS:Callias-type theorem}. 

It is a classical fact (see for instance \cite{FomenkoMoscenko80}) that the operator $D_N$ is $\tau$-Fredholm and its $\tau$-index is defined through the formula:
\[
\ind_\tau D_N\ =\ \dim_\tau\ker D^+_N\ -\ \dim_\tau\ker D^-_N.
\]
 
Let $p:N\times\RR\rightarrow N$ be the projection onto the first factor and denote by $\widehat{E}_N$ the pull-back bundle $p^\ast E_N$. Then 
\begin{equation}\label{E:widehatEN}
	\widehat{E}_N\ = \ \widehat{E}_N^+\oplus \widehat{E}_N^-,
	\qquad\text{where}\quad \widehat{E}_N^\pm:= p^*E_N^\pm.
\end{equation}
The bundle $\widehat{E}_N$ has a natural Clifford action given by:
\begin{equation}\label{E:clifford action}
	\widehat{c}(\xi,t)
	\ = \
	c(\xi)\ + \ \gamma t,\qquad  
	(\xi,t)\in T^\ast_{(x,r)} (N\oplus\RR)=
	    T^\ast_x N\oplus\RR,\qquad (x,r)\in N\times \RR,
\end{equation}
where $c$ is the Clifford action of $T^*N$ on $E_N$ and $\gamma\big|_{\widehat{E}_N^\pm}=\pm i$. Notice, however, that this action does not preserve the grading \eqref{E:widehatEN}.

Endowed with the pull-back connection $\nabla^{\widehat{E}_N}$ induced by the connection on $E_N$, the bundle $\widehat{E}_N$ becomes an ungraded Dirac $A$-Hilbert bundle. Let $\widehat{D}_N$ denote the Dirac operator associated with the connection $\nabla^{\widehat{E}_N}$ and the zero potential. With respect to the decomposition 
\begin{equation}\label{E:decomposition of sections over the cylinder}
L^2(N\times\RR,\widehat{E})=L^2(E)\tensor L^2(\RR).
\end{equation}
$\widehat{D}_N$ has the form
\begin{equation}\label{E:hatD}
	\widehat{D}_N\ :=\ D_N\otimes1\ +\ \gamma\otimes\partial_r,
\end{equation}
where $\partial_r$ denotes the operator of derivation in the axial direction of the cylinder $N\times \RR$.

Let $h:\RR\rightarrow \mathbb{R}$ be a smooth function such that
\begin{equation}\label{E:function h}
	h(r)=\left\{
	  \begin{array}{rr}-1,\qquad &r<R_1\\1\,,\qquad &r>R_2
	  \end{array}
	 \right.
\end{equation}
for some constants $R_1<R_2$.  By a slight abuse of notation we will denote by 
$h$ also the induced function $h:N\times\RR\to[-1,1]$. 
Notice that the multiplication by  $h$ is an admissible endomorphism of the ungraded Dirac $A$-Hilbert bundle $\widehat{E}_N$ (see Definition~\ref{D:ungraded admissible endomorphism}).


\begin{definition}\label{D:model operator}
The \emph{model operator} ${\bf M}$ associated with the pair $(N,E_N)$ is the Callias-type operator associated with the ungraded triple $(\widehat{E}_N,\widehat{D}_N,h)$. 
\end{definition}
Thus 
\begin{equation}\label{E:twisted Callias type operator}
	{\bf M}\ :=\ \left(\begin{array}{cc}0&{\bf M}_-\\{\bf M}_+&0\end{array}\right),
\end{equation}\label{E:model}
where ${\bf M}_\pm:=\widehat{D}_N\pm ih$.

\subsection{The index of the model operator}\label{SS:index of model}
As already seen in Section~\ref{SS:Callias-type operators}, ${\bf M}$ is $\tau$-Fredholm and its $\tau$-index is given by
\[
	\ind_\tau {\bf M}\ =\ \dim_\tau\ker{\bf M}^+-\dim_\tau\ker{\bf M}^-.
\]
Our proof of Theorem~\ref{T:computation of odd-dimensional case} is based on the following

\begin{proposition}\label{P:separation of variables in the von Neumann setting}
$\ind_\tau D_N\ =\ \ind_\tau {\bf M}$.
\end{proposition}

\begin{remark}
The definition of the operator ${\bf M}$ depends on the choice of a function $h$ satisfying \eqref{E:function h}. 
Proposition~\ref{P:separation of variables in the von Neumann setting} shows that the $\tau$-index of ${\bf M}$ is independent of this choice.
This fact justifies the notation used in Definition~\ref{D:model operator}.
\end{remark}

The rest of this section is occupied with the proof of Proposition~\ref{P:separation of variables in the von Neumann setting}.


\subsection{Differences from the case $A=\CC$.}\label{SS:diffrence from C}
When $A=\CC$, Proposition~\ref{P:separation of variables in the von Neumann setting} was proven by Anghel in \cite{Anghel1993} by using a separation of variables argument. Anghel proved that for every $\lambda>0$ there exists $\mu>0$ such that
\begin{equation}\label{E:Anghel's proof2}
	\im E_{(-\mu,\mu)}({\bf M})\ \subseteq\ \im E_{(-\lambda,\lambda)}(D_N)\tensor L^2(\RR),
\end{equation}
where $E_{(-\mu,\mu)}({\bf M})$ is the spectral projection of ${\bf M}$ relative to the interval $(-\mu,\mu)$ and $E_{(-\lambda,\lambda)}(D_N)$ is the spectral projection of $D_N$ relative to the interval $(-\lambda,\lambda)$.
Since the operator $D_N$ is Fredholm, $0$ is not in the essential spectrum of $D_N$ and we can choose $\lambda_0$ such that 
\[
	\im E_{(-\lambda_0,\lambda_0)}(D_N)\ =\ \Ker D_N.
\] 
It then follows from  \eqref{E:Anghel's proof2} that
\begin{equation}\label{E:Anghel's proof3}
	\Ker{\bf M}\ \subseteq\ \Ker D_N\tensor L^2(\RR).
\end{equation}

In the case when $A$ is an arbitrary von Neumann algebra with a finite trace $\tau$, (\ref{E:decomposition of sections over the cylinder}) and~(\ref{E:Anghel's proof2}) still hold.
However, in general the $\tau$-Fredholmness of $D_N$ doesn't imply that $0$ is not in the essential spectrum of this operator and we might not be able to deduce (\ref{E:Anghel's proof3}) from (\ref{E:Anghel's proof2}). Instead we study the square of the model operator and use the method of \cite[\S3]{Br-cob} (see also \cite[\S11.2]{Br-index}) to compute $\ker\mathbf{M}= \ker\mathbf{M}^2$.

\newcommand{\M}{\mathbf{M}}
\newcommand{\D}{\widehat{D}}\newcommand{\DN}{\widehat{D}_N}
\newcommand{\E}{\widehat{E}}\newcommand{\EN}{\widehat{E}_N}
\newcommand{\hSigma}{\widehat{\Sigma}}\newcommand{\hSigmaN}{\widehat{\Sigma}_N}
\subsection{The square of the model operator}\label{SS:square M}
Our proof of Proposition~\ref{P:separation of variables in the von Neumann setting} is based on the study of the operator
\[	
	\M^2\ = \
	\begin{pmatrix}
	 \M_-\M_+&0\\0&\M_+\M_-
	\end{pmatrix}.
\]
First we compute $\M_-\M_+$. Since $\gamma\otimes\partial_r$ anticommutes with $D_N$, multiplication by $ih$ commutes with $D_N$, $\gamma^2=-1$, and $\partial_rh-h\partial_r= h'$,  we obtain 
\begin{equation}\label{E:M-M+}
		\M_-\M_+ \ = \ 	D_N^2\otimes 1 \ - \ 1\otimes\partial_r^2 
	\ +\ i\gamma\otimes h' \ + \ 1\otimes h^2.
\end{equation}

Even though the model operator $\M$ is neither even, nor odd with respect to the grading \eqref{E:widehatEN}, it follows from the equality \eqref{E:M-M+} that $\M_-\M_+$ does preserve this grading. Moreover, 
\begin{equation}\label{E:M-M+Epm}
	\M_-\M_+\big|_{C^\infty(N\times\RR,\EN^\pm)}
	\ = \ 	
	D_N^2\otimes 1 \ + \ 1\otimes Q^\pm,
\end{equation}
where
\begin{equation}\label{E:Qpm}
	Q^\pm\ = \ -\partial_r^2 \ \mp\ h' \ + \  h^2 \ = \
	(i\partial_r\mp ih)\,(i\partial_r\pm ih) \ \ge \ 0.
\end{equation}

Similarly, 
\begin{equation}\label{E:M+M-Epm}
	\M_+\M_-\big|_{C^\infty(N\times\RR,\EN^\pm)}
	\ = \ 	
	D_N^2\otimes 1 \ + \ 1\otimes Q^\mp.
\end{equation}

\subsection{The kernel of $\M_\pm$}\label{SS:kernelM}
Since the operators $\M_+$ and $\M_-$ are adjoint of each other, we conclude that 
\begin{equation}\label{E:kerMpm}
		\ker\M_+\ = \ \ker\M_-\M_+, \qquad \ker\M_-\ = \ \ker\M_+\M_-.
\end{equation}
In particular, it follows from the discussion in Section~\ref{SS:square M} that 
\[
	\ker\M_\pm\ = \ \ker\M_\pm\big|_{C^\infty(N\times\RR,\EN^+)}
	\oplus \ker\M_\pm\big|_{C^\infty(N\times\RR,\EN^-)}.
\]

Since the operators $D_N^2$ and $Q^\pm$ are non-negative, we have
\begin{equation}\label{E:kerM2}
 \begin{aligned}
	\ker\M_-\M_+\big|_{C^\infty(N\times\RR,\EN^\pm)}\ &= \ 
	\ker D_N^2\big|_{C^\infty(N\times\RR,E_N^\pm)}
	\otimes \ker Q^\pm\ = \ \ker D_N^\pm\otimes \ker Q^\pm;\\
	\ker\M_+\M_-\big|_{C^\infty(N\times\RR,\EN^\pm)}\ &= \ 
	\ker D_N^2\big|_{C^\infty(N\times\RR,E_N^\pm)}
	\otimes \ker Q^\mp
	\ = \ \ker D_N^\pm\otimes \ker Q^\mp.
 \end{aligned}
\end{equation}
Thus to compute $\ker\M_\pm$ it remains to compute $\ker Q^\pm$. 

\begin{lemma}\label{L:kerQ}
$\dim\ker Q^+=1$ and $\dim\ker Q^-=0$.
\end{lemma}
\begin{proof}
By \eqref{E:Qpm},  $y(r)\in \ker Q^\pm\subset L^2(\RR)$ if and only if 
\[
	y'(r)\ \pm \ h(r)y(r)\ = \ 0.
\] 
The lemma follows now from the fact that the solutions of the ODE\ $y'+ h y=0$
are square-integrable, whereas the solutions of the ODE\/
$y'-  h y=0$ are not.
\end{proof}

Using \eqref{E:kerMpm} and \eqref{E:kerM2} we obtain the following corollary of  Lemma~\ref{L:kerQ}
\begin{corollary}\label{C:kerM and kerD have the same tau dimension}
The $A$-Hilbert spaces  $\ker\M_+$ and  $\ker{}D_N^+$ are isomorphic. Similarly, the $A$-Hilbert spaces $\ker\M_-$ and $\ker{}D_N^-$ are isomorphic. In particular, 
\[
	\dim_\tau\ker \M_\pm \ = \ \dim_\tau D_N^\pm. 
\]
\end{corollary}

\subsection{Proof of Proposition~\ref{P:separation of variables in the von Neumann setting}}
By Corollary~\ref{C:kerM and kerD have the same tau dimension}, we have
\[
	\ind_\tau{\bf M}=\dim_\tau\ker {\bf M}_+-\dim_\tau\ker {\bf M}_-
	=\dim_\tau \Ker D_N^+-\dim_\tau\Ker D_N^-
	=\ind_\tau D_N.
\]
\hfill $\square$


\section{A manifold with cylindrical ends}\label{S:cylindrical end}

In this section we prove Theorem~\ref{T:computation of odd-dimensional case} for the special case of a manifold with cylindrical ends. 

\subsection{A manifold with cylindrical ends}\label{SS:cylindrical end}
We use the notation of Section~\ref{SS:Callias-type theorem}. In addition we assume that $M_+= N\times[1,\infty)$ for some compact manifold $N$. In other words, we assume that 
\begin{equation}\label{E:cylindrical end}
	M\ = \ M_-\cup_N \left(N\times[1,\infty)\right),
\end{equation}
where  $M_-$ is a compact manifold with boundary, whose interior contains an essential support of the potential $\Phi\in \End_A(\Sigma)$.

\subsection{A Dirac operator on the hypersurface $N$}\label{SS:Dirac on N}
We identify $N$ with $N\times\{1\}\subset M$. Let $\Sigma_N$ and $\Phi_N$ denote the restrictions of $\Sigma$ and $\Phi$ to $N\simeq N\times\{1\}$. Since an essential support of $\Phi$ is contained in the interior of $M_-$, the bundle map $\Phi_N$ is non-degenerate.  As in Section~\ref{SS:Callias-type theorem} we decompose  
\begin{equation}\label{E:grading SigmaN}
	\Sigma_N\ =\  \Sigma_{N+}\oplus\Sigma_{N-},
\end{equation}
where, for every $x\in N$, $\Sigma_{N+}(x)$ (respectively $\Sigma_{N-}(x)$) is the closed $A$-invariant subspace of $\Sigma_N(x)$ obtained as the image of the spectral projection of $\Phi_N(x)$ corresponding to the interval $(0,\infty)$ (resp. $(-\infty,0)$).
We denote by $\Phi_{N\pm}$ the restrictions of $\Phi_N$ to $\Sigma_{N\pm}$. It follows from \eqref{E:[Phi,c]} that the grading \eqref{E:grading SigmaN} is preserved by  Clifford multiplication, i.e.
\begin{equation}\label{E:c(xi)Npm}
 	c(\xi):\,\Sigma_{N\pm}\ \to \ \Sigma_{N\pm}, 
 	\qquad \text{for all}	\quad \xi\in T^*N.
\end{equation}

The bundle $\Sigma_{N+}$ plays in what follows the same role as the bundle $E_N$ in Section~\ref{S:cylinder}. However, in general, the connection $\nabla^{\Sigma_N}$, induced on $\Sigma_N$ by $\nabla^\Sigma$, does not preserve the grading \eqref{E:grading SigmaN}. That is why we need to define new connections 
\[
	\nabla^{\Sigma_{N\pm}}\ :=\ \pr_{\Sigma_{N\pm}}\circ \nabla^{\Sigma_N}
\] 
on $\Sigma_{N\pm}$, cf. \eqref{E:nSigmaN+}. This makes the situation of this section slightly different from the one considered in Section~\ref{S:cylinder}. To apply the results of Section~\ref{S:cylinder} we need to {\em deform} the connection $\nabla^{\Sigma_N}$ so that it does preserve the grading. Lemma~\ref{L:reduction to the cylinder in the case of cylindrical end1} guarantees that such a deformation exists and preserves the index.

Let $D_{N}=D_{N+}\oplus D_{N-}$ be the  Dirac operator on $\Sigma_N$ associated to the connection $\nabla^{\Sigma_{N+}}\oplus\nabla^{\Sigma_{N-}}$ and the zero potential, cf.  \eqref{E:DN}.

\subsection{A Dirac operator on the cylinder $N\times\RR$}\label{SS:Dirac on cylinder}
Let $p:N\times\RR\to N$ be the projection and denote by 
\begin{equation}\label{E:hSigmaNpm}
		\widehat{\Sigma}_{N}\ = \ 
	\widehat{\Sigma}_{N+}\oplus \widehat{\Sigma}_{N-},
	\qquad \hSigma_{N\pm} := p^*\Sigma_{N\pm}, 
\end{equation}
the pull-back bundle over $N\times\RR$. Let $\nabla^{\hSigmaN}$ denote the connection on $\hSigmaN$ obtained by pulling back the connection $\nabla^{\Sigma_N}$. Notice that in general {\em this connection does not preserve the grading \eqref{E:hSigmaNpm}}.

Endowed with the connection $\nabla^{\hSigmaN}$ and the Clifford action \eqref{E:clifford action} the bundle $\hSigmaN$ becomes a Dirac bundle. Let 
\[
	\widehatD_N:\, C^\infty(N\times\RR,\widehat\Sigma_{N+})
	\ \to\ C^\infty(N\times\RR,\widehat\Sigma_{N+})
\]
denote the Dirac operator on $\hSigmaN$ associated to the connection $\nabla^{\hSigmaN}$ and the zero potential. 

\begin{remark}\label{R:hatDvsD}
We warn the reader that the operator $\DN$ introduced above is not related to $D_N$ by formula \eqref{E:hatD}. This is because $D_N$ was defined using the connection 
\[
	\nabla^{\Sigma_{N+}}\oplus\nabla^{\Sigma_{N-}} \ \not= \ \nabla^{\Sigma_N}.
\]
This is one of the difficulties which arise in trying to use the results of Section~\ref{S:cylinder} in the proof of Theorem~\ref{T:computation of odd-dimensional case}. We address this problem by deforming the connection $\nabla^{\Sigma_N}$ in such a way that the index of the operator $\DN$ does not change, cf. Lemma~\ref{L:reduction to the cylinder in the case of cylindrical end1}. 
\end{remark}

Finally, we denote by $\widehat{\Phi}_{N}\in \End_A(\widehat{\Sigma}_N)$ and $\widehat{\Phi}_{N\pm}\in \End_A(\widehat{\Sigma}_{N\pm})$ the lifts of $\Phi_N$ and $\Phi_{N\pm}$ to the cylinder $N\times\RR$.
Notice that $\widehat\Phi_{N+}$ is a strictly positive operator, while $\widehat\Phi_{N-}$ is a strictly negative operator.

\subsection{A Callias-type theorem for a manifold with cylindrical ends}\label{SS:Callias cylindrical end}
The main result of this section is the following special case of Theorem~\ref{T:computation of odd-dimensional case}.

\begin{proposition}\label{P:reduction to the cylinder in the case of cylindrical end}
In the situation of Theorem~\ref{T:computation of odd-dimensional case}, suppose that $M=M_-\cup_N \left(N\times[1,\infty)\right)$, where $M_-$ is a compact submanifold with boundary $\partial{M_-}\simeq N$.
Assume that there exist $\epsilon>0$ and an open neighborhood of $N$ in $M_-$ isometric to $N\times (1-\epsilon,1]$ such that $M\setminus \left(N\times(1-\epsilon, \infty)\right)$ is an essential support of $\Phi$.
Suppose that the restriction $\Sigma\big|_{N\times(1-\epsilon,\infty)}$ coincides with $\widehat{\Sigma}_{N+}\oplus\widehat{\Sigma}_{N-}$, the restriction $D\big|_{N\times(1-\epsilon,\infty)}$ coincides with $\widehat{D}_N$ 
and $\Phi\big|_{N\times(1-\epsilon,\infty)} =\ \widehat{\Phi}_{N+}\oplus\widehat{\Phi}_{N-}$.
Then 
\begin{equation}\label{E:reduction to the cylinder in the case of cylindrical end}
	\ind_\tau B_\Phi\ = \ \ind_\tau D_{N+}.
\end{equation}
\end{proposition}

The rest of this section is occupied with the proof of Proposition~\ref{P:reduction to the cylinder in the case of cylindrical end}.

\subsection{Deformation of the data on the cylindrical end}\label{SS:product on the end}
In the next two lemmas we construct a deformation of the restriction of the connection $\nabla^\Sigma$ and the potential $\Phi$ to the cylindrical end $N\times(1-\epsilon,\infty)$, and show that the $\tau$-index of $B_\Phi$ is preserved by these deformations.

In particular, in Lemma~\ref{L:reduction to the cylinder in the case of cylindrical end1} we construct a new connection $\nabla^{\prime\Sigma}$ on $\Sigma$. We denote by $D'$ the Dirac operator associated with the connection $\nabla^{\prime\Sigma}$ and the zero potential, cf. Section~\ref{SS:Callias index}.  We also denote by $\nabla^{\hSigma_{N\pm}}$ the connection on $\hSigma_{N\pm}$ obtained by pulling back the connection $\nabla^{\Sigma_{N\pm}}$.



\begin{lemma}\label{L:reduction to the cylinder in the case of cylindrical end1}
Under the hypotheses of Proposition~\ref{P:reduction to the cylinder in the case of cylindrical end}, there exists a connections $\nabla^{\prime\Sigma}$  on $\Sigma$ and a number $\lambda\ge1$, such that 
\begin{enumerate}[label=(\roman*)]
 \item \ \ 
 $\nabla^{\prime\Sigma}\big|_{N\times(1-\epsilon,\infty)}\ = \ 
	\nabla^{\widehat{\Sigma}_{N+}}\oplus \nabla^{\widehat{\Sigma}_{N-}}$;
 \item  \ \
 $\lambda\Phi$ is an admissible endomorphism for the pair $(\Sigma,D')$;
 \item \ \  $\ind_\tau B'_{\lambda\Phi}= \ind_\tau{}B_\Phi$ for all $t\in[0,1]$, where $B'_{\lambda\Phi}$ denotes the Callias-type operator associated with the triple $(\Sigma,D',\lambda\Phi)$.
\end{enumerate}
\end{lemma}

\begin{lemma}\label{L:reduction to the cylinder in the case of cylindrical end2}
Under the hypotheses of Proposition~\ref{P:reduction to the cylinder in the case of cylindrical end} assume that 
\begin{equation}\label{E:deformation of connection}
	\nabla^{\Sigma}\big|_{N\times(1-\epsilon,\infty)}\ =\ 
	\nabla^{\widehat{\Sigma}_{N+}}\oplus \nabla^{\widehat{\Sigma}_{N-}}.
\end{equation}
Then there exists an admissible endomorphism $\Phi'$ for $(\Sigma, D)$ such that 
\begin{enumerate}[label=(\roman*)]

\item 
 The restriction of $\Phi'$ to $N\times (1-\epsilon,\infty)$ is the grading operator on $\Sigma\big|_{N\times (1-\epsilon,\infty)}= \widehat{\Sigma}_{N+}\oplus\widehat{\Sigma}_{N-}$, i.e. $\Phi'\big|_{\widehat{\Sigma}_{N\pm}}=\pm1$;

\item  $\ind_\tau B_{\Phi'}= \ind_\tau B_\Phi$, where $B_{\Phi'}$ denotes the Callias-type operator associated with the triple $(\Sigma, D,\Phi')$.
\end{enumerate}
\end{lemma}

The proofs of these lemmas are presented at the end of this section, after we explained how the lemmas are used to prove  Proposition~\ref{P:reduction to the cylinder in the case of cylindrical end}.

\subsection{Scketch of the proof of Proposition~\ref{P:reduction to the cylinder in the case of cylindrical end}}\label{SS:plan proof reduction to cylinder}
If follows from Lemmas~\ref{L:reduction to the cylinder in the case of cylindrical end1} and \ref{L:reduction to the cylinder in the case of cylindrical end2} that it is enough to prove the proposition for the case when the connection $\nabla^\Sigma$ satisfies $\nabla^{\Sigma_N}= \nabla^{\Sigma_{N+}}\oplus\nabla^{\Sigma_{N-}}$ and  $\Phi_N$ is the grading operator on $\Sigma_N$.

Let $M_1=-M$ denote a copy of $M$ with the opposite orientation. Then $M_1$ is naturally isomorphic to $(N\times(-\infty,1])\cup_N(-M_-)$. The bundle $\Sigma$ induces a Dirac bundle $\Sigma_1:=-\Sigma$ on $M_1$, cf. Section~\ref{SS:opposite orientation}. Let $D_1$  be the corresponding Dirac operator. 

In Section~\ref{SS:operators on -M} we construct a potential $\Phi_1$ on $\Sigma_1$ such that 
\begin{itemize}
 \item[(i)] the restriction of the Callias-type operator $B_{\Phi_1}$ associated with the triple $(\Sigma_1,D_1,\Phi_1)$ to the cylinder $N\times(-\infty,1])$ is equal to $\mathbf{M}\oplus{T}$, where $\mathbf{M}$ is the model operator of Definition~\ref{D:model operator}, and $T^2>0$;
 \item[(ii)]  $\ind_\tau B_{\Phi_1}=0$. 
\end{itemize} 

The restriction of all the data to neighborhoods of $N\times\{1\}$ in  $M$ and $M_1$ coincide. Hence, we can apply the relative index theorem to compute 
\begin{equation}\label{E:BPhi=BPhi+}
		\ind_\tau B_\Phi\ = \ \ind_\tau B_\Phi\ +\ \ind_\tau B_{\Phi_1}.
\end{equation}
The cut and paste procedure of Section~\ref{SS:A relative index theorem} applied to manifolds $M$ and $M_1$ yields manifolds
\[
	M_2\ = \ N\times\RR \quad\text{and}\quad M_3\ = \ M_-\cup_N(-M_-).
\]
Let $B_2$ and $B_3$ be the Callias-type operators on $M_2$ and $M_3$ obtained from $B_\Phi$ and $B_{\Phi_1}$ by cutting and pasting. One readily sees that $B_2$ is equal to $\mathbf{M}\oplus{T}$, where $\mathbf{M}$ is the model operator of Definition~\ref{D:model operator}, and $T^2>0$. Hence, $\ind_\tau{}B_2= \ind_\tau\textbf{M}$. Also $B_3$ is a Callias-type  operator on a {\em compact}  manifold $M_3$. Thus\/ $\ind_\tau{}B_3=0$ by Proposition~\ref{P:compact index}. The relative index theorem and \eqref{E:BPhi=BPhi+} imply that 
\begin{equation}\label{E:indBPhi=indM}
	\ind_\tau B_\Phi\ = \ \ind_\tau B_2\ +\ \ind_\tau B_3 \ =\
	\ind_\tau \mathbf{M}.
\end{equation}
Proposition~\ref{P:reduction to the cylinder in the case of cylindrical end} follows now from Proposition~\ref{P:separation of variables in the von Neumann setting}.

\subsection{The manifold with the reversed orientation}\label{SS:opposite orientation}
Before presenting the details of the proof of Proposition~\ref{P:reduction to the cylinder in the case of cylindrical end}, we introduce some additional notation. 

For an oriented manifold $W$ we denote by $-W$ a copy of this manifold with the opposite orientation. If $E$ is a Dirac bundle over $W$, we denote by $-E$ the same bundle viewed as a vector bundle over $-W$, endowed with the {\em opposite} Clifford action. This means that a vector $\xi\in T^*W\simeq T^*(-W)$ acts on $-E$ by $c(-\xi)$. The change of the Clifford action is needed because we reversed the orientation of $M_-$  (for more details about this construction we refer to \cite[Section~2.3.2]{Bunke1995} and \cite[Chapter~9]{BooWoj93}).

From now on we assume that  $W = W_-\cup_N(N\times[1,\infty))$. Then there is a natural orientation preserving isometry  
\[
	\psi:-W\ \overset{\sim}\longrightarrow \  (N\times(-\infty,1])\cup_N(-W_-),
\]
such that 
\begin{alignat*}{2}
	\psi(x)\ &=\ x, \qquad &&\text{if}\quad x\in -W_-\simeq W_-\\ 
 	\psi(y,t)\ &=\ (y,2-t), \qquad &&\text{if}\quad (y,t)\in N\times[1,\infty).
\end{alignat*}
To simplify the notation we will skip $\psi$ from the notation and simply write 
\begin{equation}\label{E:-W}
	-W\ =\  (N\times(-\infty,1])\cup_N(-W_-).
\end{equation}

Let $E_N$ be a bundle over $N$ and let  $\widehat{E}_N$ denote the pull-back of $E_N$ by the projection map $p:N\times\RR\to N$. Suppose that $E$ is a bundle over $W = W_-\cup_N(N\times[1,\infty))$ whose restriction to $N\times[1,\infty)$ coincides with the restriction of $\widehat{E}_N$ to the same cylinder. Recall that $-E$ is a bundle over the manifold  \eqref{E:-W}. One readily sees that 
\begin{equation}\label{E:-E restricted to cylinder}
	-E\Big|_{N\times(-\infty,1]}\ \simeq\ 
	\widehat{E}_N\Big|_{N\times(-\infty,1]}.
\end{equation}

\subsection{Construction of a potential on $-M$}\label{SS:operators on -M}
Consider the manifold $M_1:=-M$. Then by \eqref{E:-W} we have 
\[
	M_1\ = \ \left(N\times (-\infty,1]\right) \cup_N (-M_-).
\] 

Consider the bundle $\Sigma_1=-\Sigma$ over $M_1$. By \eqref{E:-E restricted to cylinder} the restriction of $\Sigma_1$ to the cylinder part coincides with $\widehat{\Sigma}_N$. In particular, this restriction has a natural grading \eqref{E:hSigmaNpm}
\[
	\widehat{\Sigma}_{N}\ = \ \widehat{\Sigma}_{N+}\oplus \widehat{\Sigma}_{N-}.
\]	
We denote by $D_1$ the Dirac operator associated with the connection on  $\Sigma_1$.

Let $h:\RR\to[-1,1]$ be a smooth function such that
\begin{equation}\label{E:step2:1}
	h(r)=\left\{\begin{array}{rr}-1,\quad &r<-1\\1\qquad &r>0\\
	\end{array}\right., \qquad\qquad r\in \RR.
\end{equation}
By a slight abuse of notation we also denote by $h$ the induced function $h:N\times\RR\to [-1,1]$. Let $\Phi_1$ be the admissible endomorphism of $\Sigma_1$, which  coincides with $\Phi$ on $\left(N\times(1-\epsilon,1]\right)\cup_N(-M_-)$ and such that
\[
	\Phi_1\big|_{N\times (-\infty,1+\epsilon)}\ =\
	\left(\begin{array}{cc}h&0\\0&-1\end{array}\right).
\]
Notice, that we can view the product $N\times(1-\epsilon,1+\epsilon)$ as a subset of both manifolds $M$ and $M_1$. Then the restrictions of $\Phi$ and $\Phi_1$ to this subset coincide. 

\begin{lemma}\label{L:indPhi1=0}
Let $B_{\Phi_1}$ denote the Callias-type operator associated with the pair $(\Sigma_1,\Phi_1)$. Then 
\begin{equation}\label{E:indPhi1=0}
	\ind_\tau B_{\Phi_1}\ = \ 0.
\end{equation}
\end{lemma}
\begin{proof}
Consider a new potential $\Phi_1'\equiv -\Id$ on $\Sigma_1$ and denote by $B_{\Phi_1'}$ the corresponding Callias-type operator. Notice, that the bundle maps $\Phi_1$ and $\Phi_1'$ coincide outside of the compact set 
\[
	\big(N\times[-1,1]\big)\cup_N(-M_-).
\]
It follows now from Theorem~2.21 of \cite{BrCe15} that 
\begin{equation}\label{E:indPhi=indPhi'}
	\ind_\tau B_{\Phi_1}\ = \ \ind_\tau B_{\Phi_1'}.
\end{equation}
Since 
\[
	B_{\Phi_1'}^2\ =\ 
		  \begin{pmatrix}
	  		D_1^2+1&0\\ 0&D_1^2+1
	  \end{pmatrix} > \ 0, 
\]
we conclude that $\ind_\tau{}B_{\Phi_1'}=0$. The lemma follows now from \eqref{E:indPhi=indPhi'}.
\end{proof}

\subsection{Proof of Proposition~\ref{P:reduction to the cylinder in the case of cylindrical end}}
By  Lemmas~\ref{L:reduction to the cylinder in the case of cylindrical end1} and \ref{L:reduction to the cylinder in the case of cylindrical end2}  it is enough to prove the proposition for the case when  $\nabla^{\Sigma_N}= \nabla^{\Sigma_{N+}}\oplus\nabla^{\Sigma_{N-}}$ and  $\Phi_N$ is the grading operator on $\Sigma_N$.

Let $M_1$, $\Sigma_1$, and $\Phi_1$ be as in the previous section. As we already noted, the restrictions of $\Phi$ and $\Phi_1$ to $N\times(1-\epsilon,1+\epsilon)$ coincide. Thus we can apply the relative index theorem~\ref{T:relative index theorem in von neumann setting} to the pair of manifolds $M$ and $M_1$. As a result of the cutting and pasting procedure used in this theorem we obtain two new manifolds
\begin{equation}\label{E:M2M3}
	M_2\ :=\ N\times \RR,\qquad M_3\ :=\ M_-\cup_N (-M_-).
\end{equation}
Let $B_2$ and $B_3$ be the operators on $M_2$ and $M_3$ obtained from $B_{\Phi}$ and $B_{\Phi_1}$ by cutting and pasting. It follows immediately from our assumptions that $\Phi$ is the grading operator and from the  construction of $\Phi_1$ that the restriction of $B_2$ to $\widehat{\Sigma}_{N+}\oplus\widehat{\Sigma}_{N+}$ is equal to the model operator $\mathbf{M}$, while the restriction of $B_2$ to $\widehat{\Sigma}_{N-}\oplus\widehat{\Sigma}_{N-}$ is equal the Callias-type operator associated with the ungraded triple $(\widehat{\Sigma}_{N-},\widehat{D}_{N-},-1)$. In other words, 
\[
	B_2\big|_{\widehat{\Sigma}_{N-}\oplus\widehat{\Sigma}_{N-}}\ = \
	  \begin{pmatrix}
	  		0&\widehat{D}_{N-}+i\\ \widehat{D}_{N-}-i&0
	  \end{pmatrix}.
\]
Hence, 
\[
	B_2^2
\big|_{\widehat{\Sigma}_{N-}\oplus\widehat{\Sigma}_{N-}}\ = \
	  \begin{pmatrix}
	  		\widehat{D}_{N-}^2+1&0\\ 0&\widehat{D}_{N-}^2+1
	  \end{pmatrix} \ >\ 0.
\]
We conclude that $\ind_\tau B_2\big|_{\widehat{\Sigma}_{N-}\oplus\widehat{\Sigma}_{N-}}=0$ and 
\[
	\ind_\tau B_2\ = \ \ind_\tau \mathbf{M}.
\]

Furthermore, $B_3$ is a Callias-type operator on {\em compact manifold} $M_3$. Hence, $\ind_\tau{}B_3=0$ by Proposition~\ref{P:compact index}. Applying the relative index theorem and using Lemma~\ref{L:indPhi1=0}, we now obtain
\[
	\ind_\tau B_{\Phi}\ = \ \ind_\tau B_{\Phi}\ + \ \ind_\tau B_{\Phi_1}
	\ = \ \ind_\tau B_2\ + \ \ind_\tau B_3 \ = \ \ind_\tau \mathbf{M}.
\]
Proposition~\ref{P:reduction to the cylinder in the case of cylindrical end} follows now from Proposition~\ref{P:separation of variables in the von Neumann setting}. \hfill$\square$

\medskip
We now pass to the proofs of Lemmas~\ref{L:reduction to the cylinder in the case of cylindrical end1} and \ref{L:reduction to the cylinder in the case of cylindrical end2}.


\subsection{A rescaling of the potential}
Fix $\lambda\geq 1$.
If the endomorphism  $\Phi$ is admissible, then also $\lambda \Phi$ is admissible with the same essential support.
Denote by $B_{\lambda\Phi}$ the Callias-type operator associated with $(\Sigma, D, \lambda\Phi)$.
The hypotheses of Proposition~\ref{P:reduction to the cylinder in the case of cylindrical end} imply that $\Phi$ is constant in the axial direction over the cylindrical end.
Therefore, $\Phi$ is uniformly bounded and defines a bounded operator on $L^2(M,\Sigma)$.
It follows that $\{B_{t\Phi}\}_{1\leq t\leq \lambda}$ is a continuous homotopy of $\tau$-Fredholm operators with fixed domain.
By \cite[Lemma~7.3]{BrCe15}, we deduce 
\begin{equation}\label{E:ind B_Phi=ind B_Phi^lambda}
	\ind_\tau B_\Phi\ =\ \ind_\tau B_{\lambda\Phi}.
\end{equation}


\subsection{Proof of Lemma~\ref{L:reduction to the cylinder in the case of cylindrical end1}}
By the hypotheses of Proposition~\ref{P:reduction to the cylinder in the case of cylindrical end}, 
\[
	\Sigma\big|_{N\times (1-\epsilon,\infty)}\ =\ 
	\widehat{\Sigma}_{N+}\oplus\widehat{\Sigma}_{N-}.
\]
With respect to this decomposition, we have
\begin{equation}\label{E:reduction to the cylinder2}
    D\big|_{N\times (1-\epsilon,\infty)}\ =\
    \begin{pmatrix}
    		\widehat{D}_{N+}&
    		\widehat{\pi}_-\circ \widehat{D}_N\circ \widehat{\pi}_+
    		\\
		\widehat{\pi}_+\circ \widehat{D}_N\circ \widehat{\pi}_-& 
		\widehat{D}_{N-}
    \end{pmatrix},
\end{equation}
where $\widehat{\pi}_\pm$ are the projections onto $\widehat{\Sigma}_{N\pm}$.

Let $\widehat{\alpha}\in \End_A(\widehat{\Sigma}_N)$ be the grading operator, i.e. $\widehat{\alpha}\big|_{\widehat{\Sigma}_{N\pm}}= \pm1$. Then $\widehat{\pi}_\pm=\frac{1}{2}\left(I\pm\widehat{\alpha}\right)$ and 
\[
	\widehat{\pi}_\pm\circ \widehat{D}_N\circ \widehat{\pi}_\mp
	\ =\ 
	\pm\frac{1}{2}[\widehat{D}_N,\widehat{\alpha}]\circ \widehat{\pi}_\mp,
\] 
where $\widehat{\alpha}=\pm 1$ on $\widehat{\Sigma}_{N\pm}$. It follows from  \eqref{E:c(xi)Npm} that $\widehat{\alpha}$ commutes with the Clifford multiplication. Hence, the commutator $[\widehat{D}_N,\widehat{\alpha}]$ is a zero-order differential operator, i.e. a bundle map.  We conclude that  $\widehat{\pi}_\pm\circ \widehat{D}_N\circ \widehat{\pi}_\mp$ are  also bundle maps.

Pick a constant $\epsilon_1$ such that $0<\epsilon_1<\epsilon$ and let $\Psi$ be a self-adjoint endomorphism of $\Sigma$ such that $\Psi=0$ off $N\times (1-\epsilon,\infty)$ and 
\[
\Psi\big|_{N\times (1-\epsilon_1,\infty)}\ :=\
 \begin{pmatrix}
	0&\widehat{\pi}_-\circ \widehat{D}_N\circ \widehat{\pi}_+\\
	\widehat{\pi}_+\circ \widehat{D}_N\circ \widehat{\pi}_-& 0
 \end{pmatrix}.
\]
Since both $\Psi$ and $\Phi$ are uniformly bounded, the commutator $[\Psi,\Phi]$ is in $L^\infty(M,\End_A(\Sigma))$.
Since the restriction of $D$ to $N\times (1-\epsilon,\infty)$ is the lift of $D_N$, the commutator $[D,\Phi]$ is also in $L^\infty(M,\End_A(\Sigma))$.
Hence, we can choose constants $d>0$ and $\lambda\geq 1$ such that 
\begin{equation}\label{E:admissible lambda}
	\lambda^2\Phi^2(x)\ \geq\ d+ \lambda\left(\|[D,\Phi]\|_\infty+\|\Psi,\Phi\|_\infty\right),
	\qquad\qquad x\in N\times (1-\epsilon,\infty).
\end{equation}

Consider the family $D_t:= D-t\Psi$ of Dirac operators on $\Sigma$. We claim that if $\lambda$ satisfies \eqref{E:admissible lambda} and $t\in[0,1]$, then $\lambda\Phi$ is an admissible endomorphism for $(\Sigma,D_t)$. Indeed,  since
\[
	\left[D-t\Psi,\lambda\Phi\right]\ =\ 
	\lambda\left(\left[D,\Phi\right]-t\left[\Psi,\Phi\right]\right),
\]
by using \eqref{E:admissible lambda} we obtain
\begin{equation}\label{E:Phi is admissible for D-tPsi}
	\begin{aligned}
		\lambda^2\Phi^2(x)
		\ \geq&\quad  
		d+ \lambda\left(\|[D,\Phi]\|_\infty+\|\Psi,\Phi\|_\infty\right)
		\\ \geq&\quad
		d+\lambda\left(\|[D,\Phi]\|_\infty+t\|[\Psi,\Phi]\|_\infty\right)
		\\ \geq&\quad 
		d+\lambda\|[D,\Phi]-t[\Psi,\Phi]\|_\infty
		\\ \geq&\quad
		d+ \|[D-\lambda\Psi,\lambda\Phi](x)\|,
	\end{aligned}
\end{equation}
for every  $x\in N\times (1-\epsilon,\infty)$.

Let $B^t_{\lambda\Phi}$ denote the Callias-type operator associated with the triple $(\Sigma,D-t\Psi,\lambda{}\Phi)$.
Since $\Psi$ is uniformly bounded, the family of operators   
\[
	\big\{\,B^t_{\lambda\Phi}:\ t\in[0,1]\,\big\}
\] 
is a continuous homotopy of $\tau$-Fredholm operators with fixed domain.  By \cite[Lemma~7.3]{BrCe15}, the $\tau$-index of $B^t_{\lambda\Phi}$ is independent of $t$.  Using \eqref{E:ind B_Phi=ind B_Phi^lambda} we now deduce
\begin{equation}\label{E:ind B_3=ind D_3'+Psi}	
	\ind_\tau B_\Phi\ =\ \ind_\tau B_{\lambda\Phi}
	\ =\ \ind_\tau B^0_{\lambda\Phi}\ =\ \ind_\tau B^1_{\lambda\Phi}.
\end{equation}

Let $\nabla^{\prime\Sigma}$ be a connection on $\Sigma$ satisfying condition (i) of the lemma. In other words we assume that the restriction of $\nabla^{\prime\Sigma}$ to $N\times(1-\epsilon,\infty)$ is equal to  $\nabla^{\hSigma_{N+}}\oplus\nabla^{\hSigma_{N-}}$. Let $D'$ be the Dirac operator associated with the connection $\nabla^{\prime\Sigma}$ and the zero potential. Let $B'_{\lambda\Phi}$ denote the Callias-type operator associated with the triple $(\Sigma,D',\lambda\Phi)$. Then the restrictions of the operators $B'_{\lambda\Phi}$ and $B^1_{\lambda\Phi}$ to the  cylindrical part $N\times(1-\epsilon,\infty)$ coincide.  Hence, $\ind_\tau{}B'_{\lambda\Phi}\ = \ind_\tau{}B^1_{\lambda\Phi}$ by Theorem~2.21 of \cite{BrCe15}. The lemma follows now from \eqref{E:ind B_3=ind D_3'+Psi}.
\hfill$\square$


\subsection{Proof of Lemma~\ref{L:reduction to the cylinder in the case of cylindrical end2}}
By \eqref{E:deformation of connection}  
\[
	D\big|_{N\times (1-\epsilon,\infty)}
	\ =\ \widehat{D}_{N+}\oplus\widehat{D}_{N-},
\]
where $\D_{N\pm}$ is the Dirac operator on $\hSigma_{N\pm}$ associated with the connection $\nabla^{\hSigma_{N\pm}}$ and the zero potential. 

Pick constants $\epsilon_1$, $\epsilon_2$ such that $0<\epsilon_1<\epsilon_2<\epsilon$.
Let $f:(1-\epsilon,\infty)\rightarrow [0,1]$ be a smooth function such that $f=0$ on $(1-\epsilon,1-\epsilon_2)$ and $f=1$ on $(1-\epsilon_1,\infty)$.
Let  $\Phi'$ be a self-adjoint endomorphism of $\Sigma$ such that  $\Phi'=0$ off $N\times (1-\epsilon,\infty)$ and
\[
	\Phi'(y,r)\ =\
			\begin{pmatrix}
				f(r)&0\\0&-f(r)
			\end{pmatrix},
			\qquad\text{for}\quad (y,r)\in N\times (1-\epsilon,\infty).
\]
Observe that $[D,\Phi']=0$ outside of $N\times (1-\epsilon,\infty)$ and
\[
	[D,\Phi']\big|_{N\times (1-\epsilon,\infty)}\ =\
	\left[\,
	 \begin{pmatrix}
	   \widehat{D}_{N+}&0\\0&\widehat{D}_{N-}
	 \end{pmatrix},\,
     \begin{pmatrix}
       f&0\\0&-f
     \end{pmatrix}\,\right]
	\ =\
     \begin{pmatrix}
        [\widehat{D}_{N+},f]&0\\
        0&[f,\widehat{D}_{N-}]
     \end{pmatrix}.
\] 
Hence, $[D,\Phi']$ is a zero-order differential operator. Thus $\Phi'$ satisfies Condition~(i) of Definition~\ref{D:ungraded admissible endomorphism}.

Furthermore, 
\begin{equation}\label{E:[D,G]=0}
	\Phi'\big|_{N\times (1-\epsilon_1,\infty)}\ =\
	\begin{pmatrix}
		1&0\\0&-1
	\end{pmatrix},\qquad
	\left[D,\Phi'\right]\big|_{N\times (1-\epsilon_1,\infty)}\ =\ 0.
\end{equation}
Thus Condition~(ii) of Definition~\ref{D:ungraded admissible endomorphism} is also satisfied. We conclude that  $\Phi'$ is an admissible endomorphism for $(\Sigma, D)$ with essential support $M\setminus\left(N\times (1-\epsilon,\infty)\right)$.

To prove the lemma it remains to show that
\begin{equation}\label{E:unitary potential}
	\ind_\tau B_{\Phi'}\ =\ \ind_\tau B_\Phi,
\end{equation}
where $B_{\Phi'}$ is the Callias-type operator associated with $(\Sigma, D, \Phi')$. We prove this equality by considering the homotopy 
\begin{equation}\label{E:Phi_t}
	\Phi_t\ :=\ t\,\Phi+(1-t)\,\Phi', \qquad 0\le t\le 1.
\end{equation}
A subtlety  here is that in general the endomorphism $\Phi_t$ is not admissible.  However, we show below that there exists $\lambda\ge 1$ such that the endomorphism $\lambda\Phi_t$ is admissible for all $t\in[0,1]$. 

First, recall that in Proposition~\ref{P:reduction to the cylinder in the case of cylindrical end} we assumed that  
\[
	\Phi\big|_{N\times (1-\epsilon,\infty)}
	\ =\ 
	\widehat{\Phi}_{N+}\oplus\widehat{\Phi}_{N-}, 
\]
where $\widehat{\Phi}_{N+}= \widehat{\Phi}\big|_{\widehat{\Sigma}_{N+}}$ is a strictly positive operator and $\widehat{\Phi}_{N-}= \widehat{\Phi}\big|_{\widehat{\Sigma}_{N-}}$ is a strictly negative operator. Since $\widehat{\Phi}_{N\pm}(x)$ are constant along the axial direction on the cylinder,  there exists a constant $d_1>0$ such that $\widehat{\Phi}_{N+}\geq d_1$ and $\widehat{\Phi}_{N-}\leq -d_1$. 
It follows that
\begin{equation}\label{E:Phi^2>d_1^2}
	\Phi^2\big|_{N\times (1-\epsilon,\infty)}=
	\begin{pmatrix}
		\widehat{\Phi}_{N+}^2&0\\0&\widehat{\Phi}_{N-}^2
	\end{pmatrix}		
	\\ \geq\ d_1^2.
\end{equation}
Using \eqref{E:[D,G]=0} we also obtain
\begin{equation}\label{E:{Phi,G}}
	\big\{\Phi,\Phi'\big\}\big|_{N\times (1-\epsilon_1,\infty)}\ := \ 	
	\big(\, \Phi\circ \Phi' + \Phi'\circ\Phi\,\big)
	\big|_{N\times (1-\epsilon_1,\infty)}
	\ =\
	2\left(\begin{array}{cc}\widehat{\Phi}_{N+}&0\\0&-\widehat{\Phi}_{N-}\end{array}\right)\ \geq\ 2\,d_1.
\end{equation}

Since $\Phi'$ is constant in the axial direction over the cylindrical end, the commutator $[D,\Phi']$ is in $L^\infty(M,\End_A(\Sigma))$.
Set 
\[	
	d_2\ :=\ \min (\{d_1,d_1^2, 1\})
\] 
and pick constants $d>0$, $\lambda\geq 1$ such that
\begin{equation}\label{E:choice of d and lambda}
	\lambda^2 d_2\ \geq\ d+\lambda \|[D,\Phi']\|_\infty.
\end{equation}
We claim that $\lambda\Phi_t$ is an admissible endomorphism for all $t\in[0,1]$.  Indeed, using \eqref{E:[D,G]=0}, \eqref{E:{Phi,G}}, \eqref{E:Phi^2>d_1^2}, and the equality 
\[
	\left[D,\lambda\Phi_t\right]
	\ =\ 
	\lambda\,\big(\, t[D,\Phi]+(1-t)[D,\Phi']\,\big),
\]
we obtain
\[
\begin{aligned}
	\left(\lambda\Phi_t\right)^2(x)\ = &\ t^2\lambda^2\Phi^2(x)+(1-t)^2 \lambda^2\Phi'^2(x) + 2t(1-t)\lambda^2\{\Phi, \Phi'\}(x)\\
	\geq &\ t^2\lambda^2d_1^2+(1-t)^2 \lambda^2 + 2t(1-t)\lambda^2d_1\\
	\geq &\ \left(t^2 +(1-t)^2+ 2t(1-t)\right)\,\lambda^2\,d_2\\
	\geq &\ d+\lambda \|[D,\Phi']\|_\infty\
	\geq \ d+\lambda \|[D,\lambda \Phi_t](x)\|,
\end{aligned}
\]
for every $x\in N\times(1-\epsilon_1,\infty)$.

Let $B_{\lambda\Phi_t}$ denote the Callias-type operator associated with $(\Sigma, D, \lambda\Phi_t)$.
Since both endomorphisms, $\Phi$ and $\Phi'$, are constant in the axial direction,  the family of  endomorphisms $\Phi_t$ is uniformly bounded and depends continuously on $t$. 
By \cite[Lemma~7.3]{BrCe15}, the index $\ind_\tau{}B_{\lambda\Phi_t}$ is independent of $t\in[0,1]$. Using \eqref{E:ind B_Phi=ind B_Phi^lambda} we now obtain
\[
	\ind_\tau B_{\Phi'}\ = \ \ind_\tau B_{\lambda\Phi'}\ =\ \ind_\tau B_{\lambda\Phi}\ = \
	\ind_\tau B_\Phi.
\]
\hfill$\square$


\section{Proof of the Callias-type theorem in general case}\label{S:general case}

In this section we conclude the proof of Theorem~\ref{T:computation of odd-dimensional case}. We use the relative index theorem to reduce the computation of the index in general case to a computation on a manifold with cylindrical ends.

\subsection{Notation}\label{SS:general case notation}
Throughout the section we use the notation of Section~\ref{SS:Callias-type theorem}. In particular, $M= M_-\cup_NM_+$, $\Sigma$ is a Dirac bundle over $M$, and $D$ is a Dirac operator on $\Sigma$ associated to a connection $\nabla^\Sigma$ and a potential $V$. 

Let $\Sigma_N$ and $\Phi_N$ be the restrictions of $\Sigma$ and $\Phi$ to $N$. We denote by $\widehat{\Sigma}_N$ and $\widehat{\Phi}_N$ the lifts of $\Sigma_N$ and $\Phi_N$ to $N\times\RR$, cf. Section~\ref{SS:Dirac on cylinder}.

\subsection{Sketch of the proof of Theorem~\ref{T:computation of odd-dimensional case}}\label{SS:Sketch Callias index}
We first deform all the structures in a small neighborhood $U(N)\subset M$ of $N$ so that the following conditions hold:
\begin{itemize}
 \item[\textbf{(N.1)}] 
 $U(N)$ is isometrically diffeomorphic to $N\times(-3\epsilon,3\epsilon)$;
 
\item[\textbf{(N.2)}]
the restrictions of  $\Sigma$   and  $\widehat{\Sigma}_N$ to $N\times(-\epsilon,\epsilon)$ are isomorphic as $A$-Hilbert\ bundles with connections;

\item[\textbf{(N.3)}]  
 $\Phi\big|_{N\times(-\epsilon,\epsilon)}=\widehat{\Phi}_N\big|_{N\times(-\epsilon,\epsilon)}$;
  
\item[\textbf{(N.4)}]
 the potential $V$ vanishes on $N\times(-\epsilon,\epsilon)$;
 
 \item[\textbf{(N.5)}]
 the essential support of the Callias-type operator associated to the new structures is still contained in the interior of $M_-$. 
\end{itemize}
By \cite[Lemma~7.3]{BrCe15}, the $\tau$-index of a Callias-type operator does not change under such deformation of the data. Hence, it suffices to prove  Theorem~\ref{T:computation of odd-dimensional case} in the case when the conditions \textbf{(N.1)--(N.5)} are satisfied, which we shall henceforth assume. 

The rest of the proof is very similar to the proof of Proposition~\ref{P:reduction to the cylinder in the case of cylindrical end}. In Section~\ref{S:cylindrical end} we used the relative index theorem to reduce the computation of the index on a manifold with cylindrical ends to the computation of an index  on a cylinder. Here in exactly the same way  we use the relative index theorem to reduce the computation  of the index on $M$ to a computation of an index on a manifold with cylindrical ends. 

The rest of this section is occupied with the details of the proof of Theorem~\ref{T:computation of odd-dimensional case}.

\subsection{Deformation of the metric}\label{SS:deformation metric}
Before presenting the  construction of the deformation of the data near $N$, discussed in the beginning of Section~\ref{SS:Sketch Callias index}, let us recall how a Dirac-type operator changes under a conformal change of the Riemannian metric.  

Let $g^M$ denote the Riemannian metric on $M$. Let $c:T^*M\to \End_A(\Sigma)$, $\nabla^\Sigma$, and $\<\cdot,\cdot\>$  denote respectively the Clifford action of the cotangent bundle, the metric connection, and  the inner product on fibers of $\Sigma$.

Let $h:M\to \RR$ be a smooth compactly supported  function and define a new Riemannian metric $g_h^M:= e^{-2h}g^M$. For a cotangent vector $\xi$ we denote by $|\xi|$ and $|\xi|_h$ its norms with respect to the metrics $g^M$ and $g^M_h$ respectively. Then $|\cdot|_h = e^h|\cdot|$.

In order to make $\Sigma$ a Dirac bundle over the Riemannian manifold $(M,g^M_h)$  we also need to change the other structures. The new Clifford action $c_h:T^*M\to \End_A(\Sigma)$ is given by
\begin{equation}	\label{E:Sigmah}
 	c_h(\xi)\ := \ e^hc(\xi), \qquad \xi\in  T^*M.
\end{equation}
The formula for the new connection $\nabla_h^\Sigma$ is more involved, cf. formula (17) of \cite{Nistor99}. This formula is forced by the assumption that the connection satisfies the Leibniz rule with respect to the Clifford action.  We denote by $\Sigma_h$ the Dirac bundle over $(M,g^M_h)$ endowed with the Clifford action $c_h$ and the Clifford connection $\nabla^\Sigma_h$ (the Hermitian metric on $\Sigma_h$ is the same as on $\Sigma$). 

Let $D'$ denote the Dirac operator on $\Sigma$ associated with the connection $\nabla^\Sigma$ and the zero potential. We recall from Section~\ref{SS:Callias index} that $D-D'=V\in \End_A(\Sigma)$. Let $D'_h$ denote the Dirac operator on $\Sigma_h$ associated with the connection $\nabla^{\Sigma}_h$ and the zero potential. Then
\begin{equation}\label{E:D'h=}\notag
 	D'_h\ = \ e^{h/2}\left(\,D'\ -\ \frac{n}2\,c(dh)\,\right)\,e^{h/2}
 	\ = \ 
 	e^h\,\left(\, D'\ - \ \frac{n-1}{2}\,c(dh)\,\right).
\end{equation}
This formula was obtained by Hitchin \cite[\S1.4]{Hitchin74} (see also \cite[[\S{}II.5]{lawson1989spin}) for the case when $\Sigma$ is the bundle of spinors. The general case was treated in \cite[\S4]{Nistor99}.

We now set
\begin{equation}\label{E:Dh=}
 	D_h\ := 
 	e^h\,\left(\, D\ - \ \frac{n-1}{2}\,c(dh)\,\right).
\end{equation}
Then $D_h$ is the Dirac operator associated with connection $\nabla^\Sigma_h$ and potential $e^hV$.

More generally, let $\chi:M\to [0,1]$ be a smooth function, such that  $\chi(x)\equiv 1$ for all $x\in M\setminus{}K$, where $K$ is an essential support of $\Phi$. We consider the Dirac operator $D_{h,\chi}$ associated with the connection $\nabla^\Sigma_h$ and the potential $\chi{}e^hV$. Then 
\begin{equation}\label{E:Dhchi}
		D_{h,\chi}\ = \ D_h\ -\ (1-\chi)e^hV.
\end{equation}
 
Notice, that since $h$ and $1-\chi$ have compact support, any endomorphism $\Phi:\Sigma\to \Sigma$ which is admissible for $D$ is also admissible for $D_{h,\chi}$. We denote by $B_{\Phi,h,\chi}$ the Callias-type operator associated with the triple $(\Sigma_h,D_{h,\chi},\Phi)$. Since $B_\Phi$ and $B_{\Phi,h,\chi}$ coincide outside of a compact set, 
\begin{equation}\label{E:BPhi=BPhih}
	\ind_\tau B_{\Phi,h,\chi}\ = \ \ind_\tau B_\Phi
\end{equation}
by Theorem~2.21 of \cite{BrCe15}.

\begin{lemma}\label{L:conformal change}
Let $B_\Phi$ be a Callias-type operator associated with a triple $(\Sigma,D,\Phi)$. Assume that $K$ is an essential support of $\Phi$ with respect to the pair $(\Sigma,D)$ and that there exists a constant $d_1>0$ and a compact set $K_1\subset K$ such that $\Phi^2(x)>d_1$ for all $x\not\in K_1$. Then there exists a function $h\in C^\infty_0(M)$ such that $K_1$ is an essential support for $\Phi$ with respect to the pair $(\Sigma_h,D_{h,\chi})$.
\end{lemma}
Together with \eqref{E:BPhi=BPhih} the lemma suggests that for all index related questions we could define the essential support as the set $K_1$ such that $\Phi^2(x)>d_1$ for all $x\not\in K_1$. 

\begin{proof}
Let $h:M\to (-\infty,0]$ be a smooth compactly supported non-positive function such that 
\[
	e^{h(x)}\, \left\|\,[D,\Phi](x)-\big(1-\chi(x)\big)\,[V,\Phi](x)\,\right\| 
	\ < \ \frac{d_1}{2}, 
	\qquad \text{for all}\quad	x\in K.
\]
By \eqref{E:[Phi,c]} $\Phi$ commutes with the Clifford multiplication. Hence, using \eqref{E:Dh=} and \eqref{E:Dhchi} we obtain 
\begin{equation}\label{E:[Dhchi,Phi]}
	[D_{h,\chi},\Phi](x)\ = \ 
	e^{h(x)}\,[D,\Phi](x)\ - \ \big(1-\chi(x)\big)\,e^{h(x)}\, [V,\Phi](x),
	\qquad x\in M,
\end{equation}
and
\[
	\|[D_{h,\chi},\Phi](x)\|\ = \ 
	e^{h(x)}\, \left\|\,[D,\Phi](x)-\big(1-\chi(x)\big)\,[V,\Phi](x)\,\right\|
	\ <\ \frac{d_1}{2}.
\]
Therefore, 
\begin{equation}\label{E:Phi2>notinK1}
	\Phi(x)^2\ > \ d_1\ > \ \frac{d_1}2\ +\  \|[D_{h,\chi},\Phi](x)\|,
	\qquad\text{for all}\quad x\in K\setminus K_1.
\end{equation}

Since $e^h\le 1$ and  $1-\chi(x)=0$ for $x\not\in K$, we conclude from \eqref{E:[Dhchi,Phi]} that 
\[
	\|[D_{h,\chi},\Phi](x)\| \ \le\ \|[D,\Phi](x)\|, \qquad\text{for all}
	\quad x\not\in K.
\]
Hence, by part (ii) of Definition~\ref{D:ungraded admissible endomorphism}, there exists $d>0$ such that
\begin{equation}\label{E:Phi2>notinK}
	\Phi(x)^2\ >  \ 
	d\ +\  \|[D_{h,\chi},\Phi](x)\|, \qquad\text{for all}
	\quad x\not\in K.
\end{equation}
Set $d_2:= \min\{\frac{d_1}2,d\}$. Then combining \eqref{E:Phi2>notinK1} and \eqref{E:Phi2>notinK} we conclude that 
\[
	\Phi(x)^2\ >  \ 
	d_2\ +\  \|[D_{h,\chi},\Phi](x)\|, \qquad\text{for all}\quad x\not\in K_1.
\]
\end{proof}

\subsection{Deformation of the data in a neighborhood of $N$}\label{SS:product on collar}
By \cite[Chapter~9]{BooWoj93}, we can deform the Riemannian metric on $M$, the Clifford action of $T^*M$ on $\Sigma$ and the connection on $\Sigma$ in a small neighborhood $U(N)\subset M$ of $N$ such that Conditions \textbf{(N.1)--(N.3)} of Section~\ref{SS:Sketch Callias index} are satisfied and $\Phi(x)^2>0$ for all $x\in (N\times(-\epsilon,0])\cup_NM_+$. 

Let  $\chi:M\to [0,1]$  be a smooth function such that $\chi(x)=0$ for all $x\in N\times(-\epsilon,\epsilon)\subset U(N)$ and  $\chi(x)=1$ for all $x\not\in N\times(-2\epsilon,2\epsilon)$. We replace the potential $V$ with $\chi{}V$. Then Condition \textbf{(N.4)} of Section~\ref{SS:Sketch Callias index} is also satisfied.

It follows now from Lemma~\ref{L:conformal change} that we can deform the structures in a small neighborhood of $N$ so that with respect to the new structures the essential support of $\Phi$ is contained in the interior of $M_-\backslash(N\times(-\epsilon,0])$. Then Condition \textbf{(N.5)} of Section~\ref{SS:Sketch Callias index} is satisfied.

Since all our changes occurred only in a  relatively  compact neighborhood of $N$, it follows from  Theorem~2.21 of \cite{BrCe15} that they don't change the index of the associated Callias-type operator. Hence, it is enough to prove Theorem~\ref{T:computation of odd-dimensional case} for the case when Conditions  \textbf{(N.1)--(N.5)} of Section~\ref{SS:Sketch Callias index} are satisfied, which we will henceforth assume.


\subsection{Proof of Theorem~\ref{T:computation of odd-dimensional case}}
Let $M_1:=N\times\RR$ be the cylinder, and let $B_1= B_{\widehat{\Phi}_N}$ be the Callias-type operator on $M_1$ associated with the triple $(\widehat{\Sigma}_N,\widehat{D}_N,\widehat{\Phi}_N)$. Since the essential support of $\widehat{\Phi}_N$ is empty, $B_1^2>0$. Hence, $\ind_\tau{}B_1=0$. 

As in Section~\ref{P:reduction to the cylinder in the case of cylindrical end} we are going to cut and paste manifolds $M$ and $M_1$ along $N$ and use the relative index theorem. Notice, that we can do it, because in Section~\ref{SS:product on collar} we deformed all the data on the collar  neighborhood of $N$ in  $M$ so that now it matches the data on the cylinder $M_1$.  

Applying the cut and paste procedure of Section~\ref{SS:A relative index theorem} to manifolds $M$ and $M_1$ and potentials $\Phi$ and $\widehat{\Phi}_N$ we obtain manifolds
\[
	M_2\ =\ 
	M_-\cup_N \left(N\times [0,\infty)\right)
	\qquad\text{and}\qquad 
	 M_3\ :=\ N\times(-\infty,0]\cup_N M_+,
\]
with potentials $\Phi_2$ and $\Phi_3$ respectively. Let $B_{\Phi_2}$ and $B_{\Phi_3}$ be the corresponding Callias-type operators. 

The restriction of $\Phi_2$ to the cylindrical part $N\times [0,\infty)$ is equal to $\widehat{\Phi}_N$. Similarly, the restriction of $\Phi_3$ to $N\times(-\infty,0]$ is equal to $\widehat{\Phi}_N$. Moreover, the essential support of $\Phi_3$ is empty. Hence, $\ind_\tau{}B_3= 0$. From Proposition~\ref{P:reduction to the cylinder in the case of cylindrical end} we obtain $\ind_\tau{}B_2= \ind_\tau{}D_{N+}$.  From the relative index theorem~\ref{T:relative index theorem in von neumann setting} we now obtain
\[
	\ind_\tau B_\Phi\ = \ \ind_\tau B_\Phi\ +\ \ind_\tau B_1\ = \
	\ind_\tau B_2\ +\ \ind_\tau B_3\ = \ \ind_\tau D_{N+}.
\]
\hfill$\square$


\section{Cobordism invariance of the $\tau$-index}\label{S:cobordism}

In this section we prove Theorem~\ref{T:cobordism invariance} about the cobordism invariance of the $\tau$-index of a Callias-type operator. As a first step we give a new proof  of the cobordism invariance of the index of Dirac operators on compact manifolds. 

In this section we freely use the notation introduced in Sections~\ref{SS:cobordism} and \ref{SS:cylindrical end}--\ref{SS:Dirac on cylinder}.

\subsection{Compact cobordisms}\label{SS:compact cobordisms}
Let $\Sigma$ be a Dirac $A$-bundle over a {\em compact} manifold $M$, and let $D$ be a Dirac operator on $\Sigma$. We say that $D$ is {\em compactly null-cobordant} if there exists a null-cobordism $(W,\oSigma,\oD)$ of $D$ with $W$ a compact manifold with boundary. 

\begin{proposition}\label{P:compact cobordism}
If $D$ is compactly null-cobordant Dirac operator, then $\ind_\tau{}D=0$. 
\end{proposition}

\newcommand{\tW}{\widetilde{W}}
\newcommand{\tSigma}{\widetilde{\Sigma}}
\begin{proof}
Consider the manifold $W':= W\cup_M(M\times[0,\infty))$. Let $\Sigma'$ denote the Dirac bundle over $W'$, whose restriction to $W$ is equal to $\oSigma$ and whose restriction to the cylinder $M\times[0,\infty)$ is equal to the bundle $\widehat{\Sigma}$. Let $D'$ be the Dirac operator whose restriction to $W$ is equal to $\oD$ and whose restriction to $M\times[0,\infty)$ is equal to $\widehat{D}$.

For $\Phi'\in \End_A(\Sigma')$ we denote by  $B'_{\Phi'}$ be the Callias-type operator on $W'$ associated with the triple $(\Sigma',D',\Phi')$. 
Applying the Callias-type index theorem~\ref{T:computation of odd-dimensional case} to the operators $B'_{\Id}$ and $B'_{-\Id}$ we obtain
\[
	\ind_\tau B'_{\Id}\ = \ \ind_\tau D, \qquad 
	\ind_\tau B'_{-\Id}\ = \ 0.
\]
The proposition follows now from Corollary~\ref{C:Phi=-Phi}.
\end{proof}


\subsection{Proof of Theorem~\ref{T:cobordism invariance}}\label{SS:prcobordism invariance}
We now deduce the cobordism invariance of a Callias-type operator from Theorem~\ref{T:computation of odd-dimensional case} and Proposition~\ref{P:compact cobordism}. 

\newcommand{\oN}{\overline{N}}

Let $B_\Phi$ be the Callias-type operator associated with a triple $(\Sigma,D,\Phi)$ and let $(W,\oSigma,\oD,\oPhi)$ be a null-cobordism of $B_\Phi$. 
Choose an  open subset $\Omega\subset W$ {with compact closure} such that 
\begin{enumerate}
\item 
$\Omega$ contains the essential support of $\oPhi$; 
\item
the boundary $\overline{N}:=\partial{\Omega}$ of $\Omega$ is a smooth manifold which  intersects $M=\partial{}W$ transversely. 
\end{enumerate}
Then $\Omega\cap \partial{}W$ is an open subset of $M=\partial{}W$ which contains an essential support of $\Phi$. Furthermore
\[
	N\ :=\ \oN\cap \partial{}W\ = \ \partial\big(\,\Omega\cap M\,\big)
\]
is a smooth {compact} hypersurface in $M$. 

Let $\oSigma_{\oN}$ denote the restriction of $\oSigma$ to $\oN$. Then $\oSigma_{\oN}$ has a natural grading 
\begin{equation}\label{E:grading on oSigma}
	\oSigma_{\oN}\ = \ \oSigma_{\oN+}\oplus \oSigma_{\oN-},
\end{equation}
where the fiber of $\oSigma_{\oN+}$ (respectively $\oSigma_{\oN-}$) over $x\in \oN$ is the image of the spectral projection of $\oPhi\big|_{\oN}$ corresponding to the interval $(0,\infty)$ (respectively $(-\infty,0)$). We denote by $\oD_{\oN}$ the restriction of $\oD$ to $\oN$. It is a Dirac operator on $\oSigma_{\oN}$ which preserves the grading \eqref{E:grading on oSigma}. We denote by $\oD_{\oN\pm}$ the restriction of $\oD_{\oN}$ to $C^\infty_0(\oN,\oSigma_{\oN\pm})$.

Let $D_{N+}$ be the operator induced on $N$ by $B_\Phi$, cf. Section~\ref{SS:Callias-type theorem}. Then $(\oN,\oSigma_{\oN+},\oD_{\oN+})$ is a null-cobordism of $D_{N+}$. Hence, it follows from Proposition~\ref{P:compact cobordism}, that $\ind_\tau{}D_{N+}=0$.  We now use the Callias-type theorem~\ref{T:computation of odd-dimensional case} to obtain $\ind_\tau{}B_\phi= \ind_\tau{}D_{N+}= 0$.   \hfill$\square$


\section{The $\Gamma$-index Theorem}\label{S:Gamma index in odd dimension}

This section is devoted to the proof of the $\Gamma$-index theorem for Callias-type operators. In Section~\ref{SS:twisted Callias-type operators} we analyze Callias-type operators twisted by an $A$-Hilbert bundle of finite type.
In section~\ref{SS:Galois covers} we show that the  Callias-type operators lifted to Galois covers can be interpreted using the twisted construction and deduce from this Lemma~\ref{L:Gamma index of Callias-type operators} and Theorem~\ref{T:Gamma-index theorem in the odd-dimensional case}.


\subsection{Twisted Callias-type operators}\label{SS:twisted Callias-type operators}

Let $M$, $S$, $D$ and $\Phi$ be as in Section~\ref{SS:the Gamma-index theorem}. We denote by $\nabla^S$ the connection on $S$.
Recall from \eqref{E:definition of Dirac operator} that $D$ is  the Dirac operator on $S$ given by
\begin{equation}\label{E:Dirac operator}
		D\ = \ \sum_i\,c(e^i)\,\nabla^{S}_{e_i} \ + \ V,
\end{equation}
where $V\in \End_A(S)$ is a bundle map.
 
Suppose that $H$ is an $A$-Hilbert bundle of finite type endowed with a connection $\nabla^H$. Then the bundle $S\tensor H$ carries a Dirac $A$-Hilbert bundle structure with connection
\[
	\nabla^{S\otimes{}H}\ := \ 
	\nabla^S\tensor 1\ +\ 1\tensor\nabla^H.
\]
We define a {\em twisted Dirac operator} on $S\otimes{}H$ by 
\begin{equation}\label{E:twisted Dirac operator}
		D_H\ = \ \sum_i\,c(e^i)\,\nabla^{S\otimes H}_{e_i} \ + \ V\otimes 1.
\end{equation}

\begin{lemma}\label{L:PhiH is admisssible}
The endomorphism 
\[
	\Phi_H\ := \ \Phi\otimes 1\ \in \ \End_A\big(S\otimes H\big),
\]
is admissible for the pair $(S\otimes{}H,D_H)$.
\end{lemma}

\begin{proof}

Fix a trivializing neighborhood $U$ of $S\tensor H$ and local sections $s\in C^\infty(U,S|_U)$, $h\in C^\infty(U,H|_U)$. 
By (\ref{E:Dirac operator}) and (\ref{E:twisted Dirac operator}) we have
\[
	D_H(s\tensor h)
	\ =\ \left(Ds\right)\tensor h\ +\ 
	\sum_i\left(c(X^i) s\right)\tensor\left(\nabla_{X_i}^H h\right),
\]
where $\{X_i\}$ is a local orthonormal frame of $TM$ and $X^i$ is the dual frame of $T^*M$. It follows that
\begin{equation}\label{E:twisted endomorphism}
\left[D_H,\Phi_H\right](s\tensor h)=\left(\left[D,\Phi\right]s\right)\tensor h+\sum_i\left(\left[c(X_i),\Phi\right]s\right)\tensor\left(\nabla_{X_i}^H h\right).
\end{equation}
By \eqref{E:[Phi,c]}, the endomorphism  $\Phi$  commutes with Clifford multiplication.  Hence the second term on the right-hand side of \eqref{E:twisted endomorphism} vanishes. Therefore,
\[
	\left[D_H,\Phi_H\right]\ =\ \left[D,\Phi\right]\tensor 1.
\]
It follows that $\Phi_H$ is admissible whenever $\Phi$ is admissible.
\end{proof}

Let $B^H_\Phi$ denote the Callias-type operator associated with the triple $(S\otimes{}H, D_H, \Phi_H)$. When the connection $\nabla^H$ is flat, Theorem~\ref{T:computation of odd-dimensional case} allows to connect the $\ind_\tau B_\Phi^H$ and $\ind B_\Phi$ in a particularly nice way.

\begin{theorem}\label{T:twisted index theorem for flat bundles}
Let $S$ be an ungraded Dirac bundle over a complete odd-dimensional oriented Riemannian manifold $M$ and let $\Phi$ be an admissible self-adjoint endomorphism of $S$. Suppose $H\rightarrow M$ is a flat $A$-Hilbert bundle of finite type.
Then
\[
	\ind_\tau B^H_\Phi\ =\ d\cdot\ind B_\Phi,
\]
where $d$ is the $\tau$-dimension of the typical fiber of $H$.
\end{theorem}

\begin{proof}
Choose a compact hypersurface $N\subset M$ such that $M=M_-\cup_N M_+$, where $M_-$ is compact and contains an essential support of both endomorphisms $\Phi$ and $\Phi_H$.

We apply the construction of Section~\ref{SS:Callias-type theorem} to construct the Dirac bundles $S_{N+}$ and $\left(S\tensor H\right)_{N+}$ over $N$. Let $D_{N+}$ and  $D_{N+}^H$ the Dirac operators on $N$ defined as in \eqref{E:DN}.
By Theorem~\ref{T:computation of odd-dimensional case}, we have
\begin{equation}\label{E:twisted Gamma-index1}
	\ind_\tau B^H_\Phi\ =\ \ind_\tau D_{N+}^H,
	\qquad  \ind B_\Phi=\ind D_{N+}.
\end{equation}

Let $H_N\rightarrow N$ denote the restriction of  the flat bundle $H$ to $N$ and let $D_{N+}^{H_N}$ denote the Dirac operator $D_{N+}$ twisted with the bundle $H_N$.
Observe that $\left(S\tensor H\right)_{N+}=S_{N+}\tensor H_N$ so that  we can identify $D_{N+}^H$ with the operator $D_{N+}^{H_N}$.
Therefore,
\begin{equation}\label{E:twisted Gamma-index2}
\ind_\tau D_N^H=\ind_\tau D_{N+}^{H_N}.
\end{equation}
Finally, since $H_{N+}$ is a flat $A$-Hilbert bundle, from \cite[Theorem~7.30 and Corollary~5.13]{Schick05l2} we get
\begin{equation}\label{E:twisted Gamma-index3}
\ind_\tau D_{N+}^{H_N}=d\cdot \ind (D_{N+}),
\end{equation}
where $d$ is the $\tau$-dimension of the typical fiber of $H_N$, that by definition of this bundle coincides with the typical fiber of $H$.
The theorem follows now from equations (\ref{E:twisted Gamma-index1}), (\ref{E:twisted Gamma-index2}) and (\ref{E:twisted Gamma-index3}).
\end{proof}


\subsection{Galois covers}\label{SS:Galois covers}

Suppose $\Gamma$ is a discrete group and denote by $l^2(\Gamma)$ the Hilbert space  of complex valued square summable functions on $\Gamma$.
We let $\Gamma$ act on the Hilbert space $l^2(\Gamma)$ by the \emph{right regular representation}
\begin{equation}\label{E:right regular representation}
\left(R_g f\right)(h):=f\left(h\cdot g\right),\qquad\ g,h\in\Gamma,\ \ f\in l^2(\Gamma).
\end{equation}
Observe that this action induces an action of the group algebra $\CC\Gamma$ on $l^2(\Gamma)$, that coincides with the right convolution multiplication.
Observe also that the operator $R_g$ defined by formula (\ref{E:right regular representation}) is bounded and that $R_g^\ast=R_{g^{-1}}$.
In this way we identify the group algebra $\CC\Gamma$ with a $\ast$-subalgebra of $\mathcal{B}(l^2(\Gamma))$.
The weak closure of $\CC\Gamma$ in $\mathcal{B}(l^2(\Gamma))$ is called the {\em group von Neumann algebra} of $\Gamma$ and is denoted by $\NGam$. On this algebra we have the canonical faithful positive trace $\tau$ defined by
\begin{equation}\label{E:canonical trace}
\tau(f)=\innerprod{f(\delta_e),\delta_e}_{l^2(\Gamma)},\ \ \ \ \ \ \ \ \ \ \ \ \ \ \ \ f\in \NGam,
\end{equation}
where $\delta_e\in l^2(\Gamma)$ is by definition the characteristic function of the unit element.

Notice that the right $\Gamma$-action on $l^2(\Gamma)$ extends to a right $\NGam$-action.
In this way we endow the space $l^2(\Gamma)$ with a Hilbert $\NGam$-space structure.
We also let  $\Gamma$ act on the Hilbert space $l^2(\Gamma)$ by the \emph{left regular representation}
\begin{equation}\label{E:left regular representation}
\left(L_g f\right)(h):=f\left(g^{-1}\cdot h\right),\ \ \ \ \ \ \ \ \ \ \ \ \ \ \ \ g,h\in\Gamma,\ \ f\in l^2(\Gamma).
\end{equation}
The right action of $\NGam$ on $l^2(\Gamma)$ commutes with the left $\Gamma$-action. 
Therefore, 
\[
	H_\Gamma\ :=\ \widetilde M\times_\Gamma l^2(\Gamma)
\] 
is an $\NGam$-Hilbert bundle of finite type on $M$.

It follows from  \eqref{E:canonical trace} that the $\tau$-dimension of $l^2(\Gamma)$ is one. Hence, the $\tau$-dimension of the typical fiber of $H_\Gamma$ is also 1, i.e.
\begin{equation}\label{E:dim_tau H_Gamma}
	\dim_\tau H_{\Gamma,x}\ = \ 1, \qquad x\in M.
\end{equation}

\subsection{Lift of a Callias-type operator to a Galois cover}\label{SS:tildeBPhi}

Notice that, since $\Gamma$ is discrete, $H_\Gamma$ is endowed with a canonical flat connection $\nabla^{H_\Gamma}$. Let $D_{H_\Gamma}$ and $\Phi_{H_\Gamma}$ be the twisted Dirac operator and the endomorphism induced in $S\otimes{}H_\Gamma$ by $D$ and $\Phi$ as in Section~\ref{SS:twisted Callias-type operators}.

Let $M$, $S$, $\Phi$, $B_\Phi$, $\widetilde{M}$, $\Gamma$, $\widetilde{S}$, $\widetilde{\Phi}$, $\widetilde{B_\Phi}$ be as in Section~\ref{SS:the Gamma-index theorem}.
Notice that
\[
	\widetilde{B_\Phi}\ =\ 
	\begin{pmatrix}
	  0&\widetilde{D}-i\widetilde{\Phi}\\\widetilde{D}+i\widetilde{\Phi}&0
	\end{pmatrix},
\]
where $\widetilde{D}$ and $\widetilde{\Phi}$ are the lifts of $D$ and $\Phi$ to the Galois cover.
We want to compare the operator $\widetilde{B_\Phi}$ with an operator $B^H_\Phi$ given by the twisted construction of Section~\ref{SS:twisted Callias-type operators}.

Observe that the action of $\Gamma$ on $L^2(\widetilde{M},\widetilde{S})$ induces an $\NGam$-Hilbert space structure on this space.
Observe also that a closed subspace $L\subset L^2(\widetilde{M},\widetilde{S})$ is $\Gamma$-invariant if and only if it is $\NGam$-invariant and in this case the $\tau$-dimension and the $\Gamma$-dimension of $L$ coincide.  
Moreover, there is a $\NGam$-Hilbert space isomorphism
\begin{equation}
	\mathcal{I}:\,L^2(\widetilde{M},\widetilde{S})
	\ \longrightarrow \ L^2(M,S\tensor H_\Gamma)
\end{equation}
with the following properties:
\begin{enumerate}
\item [(a)] For every $s\in C^\infty_c(\widetilde{M},\widetilde{S})$
\begin{equation}\label{E:formula of isomorphism of Hilbert Gamma modules}
	\left(\mathcal{I}s\right)(x)
	\ =\ 
	\sum_{\gamma\in\Gamma}s(\gamma\cdot \widetilde{x})\tensor 
	(\widetilde{x},\gamma),\qquad  x\in M,
\end{equation}
where $\widetilde{x}$ is any lift of $x$ to $\widetilde{M}$.
Notice that in Equation~(\ref{E:formula of isomorphism of Hilbert Gamma modules}) the fibers $\widetilde{S}_{\gamma\widetilde{x}}$ and $S_x$ are identified.
\item [(b)] The operators $\widetilde{D}$ and $D_{H_\Gamma}$ are conjugated through $\mathcal{I}$. Here, $\widetilde{D}$ is the lifting of $D$ to $\widetilde{M}$ and $D_{H_\Gamma}$ is the Dirac operator associated with the twisted bundle $S\tensor H_\Gamma$: cf. Formula~\eqref{E:twisted Dirac operator}.
\item[(c)] If $L$ is a closed $\NGam$-invariant subspace of $L^2(\widetilde{M},\widetilde{S})$, then $\dim_\Gamma L=\dim_\tau \mathcal{I}(L)$.
\item[(d)] Suppose that
\[
	\mathcal{D}_1\ =\ 
	\begin{pmatrix}
		0&D_1^-\\D_1^+&0
	\end{pmatrix} \qquad\text{and}\qquad
	\mathcal{D}_2\ =\ 
	\begin{pmatrix}
		0&D_2^-\\D_2^+&0
	\end{pmatrix}
\]
are odd formally self-adjoint $\NGam$-equivariant differential operators  acting respectively on $C^\infty_c(\widetilde{M},\widetilde{S}\oplus\widetilde{S})$ and $C^\infty_c(M,(S\oplus S)\tensor H)$.
Suppose also that $\mathcal{I}\circ D_1^\pm=D_2^\pm\circ\mathcal{I}$ on $C^\infty_c(\widetilde{M},\widetilde{S})$.
Then $\mathcal{D}_1$ is essentially self-adjoint if and only if $\mathcal{D}_2$ is and $\mathcal{D}_1$ is $\tau$-Fredholm if and only if $\mathcal{D}_2$ is.
In this case, it follows from (c)  that $\ind_\tau \mathcal{D}_1$ and $\ind_\tau \mathcal{D}_2$ coincide.
\end{enumerate}
For more details about the construction of the map $\mathcal{I}$ and its properties, we refer to \cite[Section~7.5]{Schick05l2}, where the case when $M$ is compact is treated. 
The case when $M$ is noncompact follows with minor modifications.

\begin{lemma}\label{L:tau-fredholmness of B lifted to Galois covers}
The operator $\widetilde{B_\Phi}$ is $\tau$-Fredholm and we have
\begin{equation}
	\ind_\Gamma\widetilde{B_\Phi}\ =\ \ind_\tau B_\Phi^{H_\Gamma}.
\end{equation}
\end{lemma}

\begin{proof}
By \cite{BrCe15} (see also Section~\ref{SS:Callias-type operators} of the present paper), the operator $B^{H_\Gamma}_\Phi$ is essentially self-adjoint and its closure is $\tau$-Fredholm.
By points (b) and (d), to prove the lemma  it suffices to show that $\Phi_H\circ\mathcal{I}=\mathcal{I}\circ \widetilde{\Phi}$ on $C^\infty_c(\widetilde{M},\widetilde{S})$.
Fix $s\in C^\infty_c(\widetilde{M},\widetilde{S})$. By point (a) we have
\[
\begin{array}{rcl}
\left(\mathcal{I}\circ \widetilde{\Phi}\right)(s)(x)&=&\sum_{\gamma\in\Gamma}\widetilde{\Phi}(s(\gamma\cdot\widetilde{x}))\tensor(\widetilde{x},\gamma)
=\sum_{\gamma\in\Gamma}\widetilde{\Phi}(s(\gamma\cdot\widetilde{x}))\tensor(\widetilde{x},\gamma)\vspace{0.2cm}\\
&=&\widetilde\Phi_H\left(\sum_{\gamma\in\Gamma}s(\gamma\cdot\widetilde{x})\tensor(\widetilde{x},\gamma)\right)
=\left( \Phi_H \circ\mathcal{I} \right)(s)(x),
\end{array}
\]
where $\widetilde{x}$ is any lift of $x$ to $\widetilde{M}$. The proof is complete.
\end{proof}

\subsection{Proof of Lemma~\ref{L:Gamma index of Callias-type operators}}
It follows from Lemma~\ref{L:tau-fredholmness of B lifted to Galois covers} and the fact that the $\tau$-dimension and the $\Gamma$-dimension coincide on closed $\NGam$-invariant subspaces of $L^2(\widetilde{M},\widetilde{S})$.
\hfill $\square$

\subsection{Proof of Theorem~\ref{T:Gamma-index theorem in the odd-dimensional case}}
Since $H_\Gamma$ is a flat $\NGam$-bundle, the thesis follows from \eqref{E:dim_tau H_Gamma},  Lemma~\ref{L:tau-fredholmness of B lifted to Galois covers}, and Theorem~\ref{T:twisted index theorem for flat bundles}.
\hfill$\square$


\end{document}